\RequirePackage{fix-cm}
\documentclass[nospthms]{svjour3}
\smartqed  
\usepackage[T1]{fontenc}

\usepackage{mathptmx}
\usepackage{makeidx}

\usepackage{amsmath}
\usepackage{amsfonts}
\usepackage{amsthm}

\usepackage{cite}
\usepackage{tikz}
\usepackage{graphicx}
\usepackage{subfig}
\usepackage[shortlabels]{enumitem}

\newcommand{\cC}{C}
\newcommand{\cD}{\mathcal{D}}

\newcommand{\cG}{\mathcal{G}}

\newcommand{\cL}{L}

\newcommand{\cW}{W}

\newcommand{\cF}{\mathcal{F}}
\newcommand{\cZD}{\mathcal{ZD}}

\newcommand{\fB}{\mathfrak{B}}

\usetikzlibrary{arrows}

\newtheorem{proposition}{Proposition}[section]

\newtheorem{theorem}[proposition]{Theorem}

\theoremstyle{definition}
\newtheorem{definition}[proposition]{Definition}
\newtheorem{remark}[proposition]{Remark}

\newtheorem{assumption}[proposition]{Assumption}

\newcommand{\N}{{\mathbb{N}}}
\newcommand{\Z}{{\mathbb{Z}}}
\newcommand{\R}{{\mathbb{R}}}
\newcommand{\C}{{\mathbb{C}}}

\newcommand{\T}{{\mathbb{T}}}

\newcommand{\Gl}{\mathbf{Gl}}

\DeclareMathOperator{\RE}{Re}

\DeclareMathOperator{\im}{im}

\DeclareMathOperator{\rk}{rank}
\DeclareMathOperator{\diag}{diag}

\newcommand{\ddt}{\frac{\text{\normalfont d}}{\text{\normalfont d}t}}
\newcommand{\ddts}{\tfrac{\text{\normalfont d}}{\text{\normalfont d}t}}
\newcommand{\ds}[1]{{\rm \, d} #1 \,}
\newcommand{\setdef}[2]{\left\{\ #1\ \left|\ \vphantom{#1} #2\ \right.\right\}}

\DeclareMathOperator{\loc}{loc}
\DeclareMathOperator{\rmref}{ref}

\newlength{\innersep}
\newlength{\maxlength}
\newlength{\dummylength}

\newcommand{\JordanBlock}[3]{
\setlength{\arraycolsep}{0pt}
\renewcommand{\arraystretch}{0}
\settowidth{\maxlength}{$#1$}
\settoheight{\dummylength}{$#1$}
\ifdim\dummylength>\maxlength
  \setlength{\maxlength}{\dummylength}
\fi
\settowidth{\dummylength}{$#2$}
\ifdim\dummylength>\maxlength
  \setlength{\maxlength}{\dummylength}
\fi
\settoheight{\dummylength}{$#2$}
\ifdim\dummylength>\maxlength
  \setlength{\maxlength}{\dummylength}
\fi
\setlength{\innersep}{0.1\maxlength}
\addtolength{\maxlength}{\innersep}
\addtolength{\maxlength}{\innersep}
\newcommand{\invisiblebox}{\phantom{\rule{\maxlength}{\maxlength}}}
\begin{array}{ccc}
  {\tikz[remember picture] \node[outer sep=0,inner sep=\innersep] (a11) {$#1$};} & \invisiblebox & {\tikz[remember picture] \node[outer sep=0,inner sep=\innersep] (a13) {$#1$};}\\
   \invisiblebox &\phantom{\rule{#3}{#3}} & \invisiblebox \\
   {\tikz[remember picture] \node[outer sep=0,inner sep=\innersep] (a31) {$#2$};} & \invisiblebox & {\tikz[remember picture] \node[outer sep=0,inner sep=\innersep] (a33) {$#1$};}
\end{array}
\tikz[remember picture, overlay] \draw (a11) edge[very thick] (a33);
}

\newcommand{\UnityBlock}[3]{
\setlength{\arraycolsep}{0pt}
\renewcommand{\arraystretch}{0}
\settowidth{\maxlength}{$#1$}
\settoheight{\dummylength}{$#1$}
\ifdim\dummylength>\maxlength
  \setlength{\maxlength}{\dummylength}
\fi
\settowidth{\dummylength}{$#2$}
\ifdim\dummylength>\maxlength
  \setlength{\maxlength}{\dummylength}
\fi
\settoheight{\dummylength}{$#2$}
\ifdim\dummylength>\maxlength
  \setlength{\maxlength}{\dummylength}
\fi
\setlength{\innersep}{0.1\maxlength}
\addtolength{\maxlength}{\innersep}
\addtolength{\maxlength}{\innersep}
\newcommand{\invisiblebox}{\phantom{\rule{\maxlength}{\maxlength}}}
\begin{array}{ccc}
  {\tikz[remember picture] \node[outer sep=0,inner sep=\innersep] (a11) {$#1$};} & \invisiblebox & \invisiblebox\\
   \invisiblebox &\phantom{\rule{#3}{#3}} & \invisiblebox \\
   \invisiblebox & \invisiblebox & {\tikz[remember picture] \node[outer sep=0,inner sep=\innersep] (a33) {$#1$};}
\end{array}
\tikz[remember picture, overlay] \draw (a11) edge[very thick] (a33);
}

\newcommand{\ReverseUnityBlock}[3]{
\setlength{\arraycolsep}{0pt}
\renewcommand{\arraystretch}{0}
\settowidth{\maxlength}{$#1$}
\settoheight{\dummylength}{$#1$}
\ifdim\dummylength>\maxlength
  \setlength{\maxlength}{\dummylength}
\fi
\settowidth{\dummylength}{$#2$}
\ifdim\dummylength>\maxlength
  \setlength{\maxlength}{\dummylength}
\fi
\settoheight{\dummylength}{$#2$}
\ifdim\dummylength>\maxlength
  \setlength{\maxlength}{\dummylength}
\fi
\setlength{\innersep}{0.1\maxlength}
\addtolength{\maxlength}{\innersep}
\addtolength{\maxlength}{\innersep}
\newcommand{\invisiblebox}{\phantom{\rule{\maxlength}{\maxlength}}}
\begin{array}{ccc}
   \invisiblebox & \invisiblebox & {\tikz[remember picture] \node[outer sep=0,inner sep=\innersep] (a13) {$#1$};}\\
   \invisiblebox &\phantom{\rule{#3}{#3}} & \invisiblebox \\
   {\tikz[remember picture] \node[outer sep=0,inner sep=\innersep] (a31) {$#1$};} &\invisiblebox & \invisiblebox
\end{array}
\tikz[remember picture, overlay] \draw (a13) edge[very thick] (a31);
}

\newcommand{\LowerNilBlock}[3]{
\setlength{\arraycolsep}{0pt}
\renewcommand{\arraystretch}{0}
\settowidth{\maxlength}{$#1$}
\settoheight{\dummylength}{$#1$}
\ifdim\dummylength>\maxlength
  \setlength{\maxlength}{\dummylength}
\fi
\settowidth{\dummylength}{$#2$}
\ifdim\dummylength>\maxlength
  \setlength{\maxlength}{\dummylength}
\fi
\settoheight{\dummylength}{$#2$}
\ifdim\dummylength>\maxlength
  \setlength{\maxlength}{\dummylength}
\fi
\setlength{\innersep}{0.1\maxlength}
\addtolength{\maxlength}{\innersep}
\addtolength{\maxlength}{\innersep}
\newcommand{\invisiblebox}{\phantom{\rule{\maxlength}{\maxlength}}}
\begin{array}{cccc}
  \tikz[remember picture] \node[outer sep=0,inner sep=\innersep] (a11) {$#1$}; &  & \invisiblebox & \invisiblebox\\
  {\tikz[remember picture] \node[outer sep=0,inner sep=\innersep] (a21) {$#2$};}&  & \invisiblebox & \invisiblebox\\
   & \phantom{\rule{#3}{#3}} &  & \\
   \invisiblebox & &  {\tikz[remember picture] \node[outer sep=0,inner sep=\innersep] (a43) {$#2$};} &
   {\tikz[remember picture] \node[outer sep=0,inner sep=\innersep] (a44) {$#1$};}
\end{array}
\tikz[remember picture, overlay] \draw (a11) edge[very thick] (a44);
\tikz[remember picture, overlay] \draw (a21) edge[very thick] (a43);
}

\newcommand{\SmallNilBlock}[3]{
\LowerNilBlock{\mbox{\scriptsize $#1$}}{\mbox{\scriptsize $#2$}}{#3}
}

\newcommand{\RectBlock}[3]{
\setlength{\arraycolsep}{0pt}
\renewcommand{\arraystretch}{0}
\settowidth{\maxlength}{$#1$}
\settoheight{\dummylength}{$#1$}
\ifdim\dummylength>\maxlength
  \setlength{\maxlength}{\dummylength}
\fi
\settowidth{\dummylength}{$#2$}
\ifdim\dummylength>\maxlength
  \setlength{\maxlength}{\dummylength}
\fi
\settoheight{\dummylength}{$#2$}
\ifdim\dummylength>\maxlength
  \setlength{\maxlength}{\dummylength}
\fi
\setlength{\innersep}{0.1\maxlength}
\addtolength{\maxlength}{\innersep}
\addtolength{\maxlength}{\innersep}
\newcommand{\invisiblebox}{\phantom{\rule{\maxlength}{\maxlength}}}
\begin{array}{ccccc}
  \tikz[remember picture] \node[outer sep=0,inner sep=\innersep] (a11) {$#1$}; &{\tikz[remember picture] \node[outer sep=0,inner sep=\innersep] (a21) {$#2$};}& & \invisiblebox & \invisiblebox\\
  &&\phantom{\rule{#3}{#3}} &&\\
   \invisiblebox &\invisiblebox& &
   {\tikz[remember picture] \node[outer sep=0,inner sep=\innersep] (a44) {$#1$};} & {\tikz[remember picture] \node[outer sep=0,inner sep=\innersep] (a43) {$#2$};}
\end{array}
\tikz[remember picture, overlay] \draw (a11) edge[very thick] (a44);
\tikz[remember picture, overlay] \draw (a21) edge[very thick] (a43);
}

{\left(\begin{smallmatrix}}
{\end{smallmatrix}\right)}

\newenvironment{smallbmatrix}
{\left[\begin{smallmatrix}}
{\end{smallmatrix}\right]}

\makeatletter
\renewcommand*\env@matrix[1][*\c@MaxMatrixCols c]{%
  \hskip -\arraycolsep
  \let\@ifnextchar\new@ifnextchar
  \array{#1}}
\makeatother

 \journalname{DAE Forum}

\makeindex

\begin{document}

\title{Vector relative degree and funnel control for differential-algebraic systems\thanks{This work was supported by the German Research Foundation (Deutsche Forschungsgemeinschaft) via the grant BE 6263/1-1.}}


\titlerunning{Vector relative degree and funnel control}        

\author{Thomas Berger \and Huy Ho\`{a}ng L\^{e} \and Timo Reis}


\institute{Corresponding author: Thomas Berger  \at Tel.: +49 5251 60-3779
            \and
            Thomas Berger \at
              Institut f\"ur Mathematik, Universit\"at Paderborn, Warburger Str.~100, 33098~Paderborn, Germany \\
              \email{thomas.berger@math.upb.de}
           \and
           Huy Ho\`{a}ng L\^{e} \at
             Department of Mathematics, National University of Civil Engineering, 55 Giai Phong road, Dong Tam, Hai Ba Trung,
              Hanoi, Vietnam\\
              \email{hoanglh@nuce.edu.vn}           
           \and
           Timo Reis \at
           Fachbereich Mathematik, Universit\"at Hamburg, Bundesstr. 55,  20146 Hamburg, Germany\\
              \email{timo.reis@uni-hamburg.de}           
           }

\date{Received: date / Accepted: date}

\maketitle

\begin{abstract}
We consider tracking control for multi-input multi-output differen\-tial-al\-ge\-braic systems. First, the concept of vector relative degree is generalized for linear systems and we arrive at the novel concept of ``truncated vector relative degree'', and we derive a~new normal form. Thereafter, we consider a~class of nonlinear functional differen\-tial-algebraic systems which comprises linear systems with truncated vector relative degree. For this class we introduce a~feedback controller which achieves that, for a~given sufficiently smooth reference signal, the tracking error evolves within a prespecified performance funnel. We illustrate our results by an example of a robotic manipulator.

\keywords{adaptive control \and differential-algebraic equations \and funnel control \and relative degree
}
 \subclass{
        34A09  
\and   	93C05  
\and	93C10  
\and	93C23  
\and 	93C40  
}
\end{abstract}

\section{Introduction}

Funnel control has been introduced in \cite{IlchRyan02b} almost two decades ago. Meanwhile, plenty of articles have been published in which funnel control from both a theoretical and an applied perspective are considered, see e.g.~\cite{BergLe18a,IlchRyan08,HackHopf13,LibeTren13b,Berg16b,Hack17,BergOtto19,SenfPaug14,BergReis14a,BergRaue18} to mention only a~few.

A typical assumption in funnel control is that the system has a~{\em strict relative degree}, which means that the input-output behavior can be described by a differential equation which has the same order for all outputs. However, multi-input, multi-output systems that appear in real-world applications do not always have a~strict relative degree. Instead, the input-output behavior is described by a collection of differential equations of different order for each output, which is referred to as {\em vector relative degree}.

The subject of this article twofold: First we consider linear (not necessarily regular) systems described by differential-algebraic equations (DAEs). We generalize the notion of vector relative degree as given in~\cite[Def.~5.3.4]{Berg14a} for regular DAEs, see~\cite{Muel09a,Isid95} for systems of ordinary differential equations (ODEs). Furthermore, we develop a~normal form for linear DAE systems which allows to read off this new \emph{truncated} vector relative degree as well as the zero dynamics. Thereafter, we consider a~class of nonlinear functional DAE systems which encompasses linear systems in this normal form, and we introduce a~new funnel controller for this system class.

Our results generalize, on the one hand, the results
of~\cite{BergLe18a}, where systems with strict relative degree are considered. On the other hand, concerning funnel control, the results in this article generalize those of~\cite{Berg16b,BergIlch14} for linear and nonlinear DAEs, where the truncated vector relative degree (although this notion does not appear in these articles) is restricted to be component-wise less or equal to one. Note that~\cite{Berg16b} already encompasses the results found in~\cite{BergIlch12a} for linear DAE systems with properly invertible transfer function. DAEs with higher relative degree have been considered in~\cite{BergIlch12b}, and even this article is comprised by the present results. Therefore, the present article can be seen as a unification of the funnel control results presented in the previous works~\cite{Berg16b,BergIlch12a,BergIlch12b,BergIlch14,BergLe18a} to a fairly general class of nonlinear DAE systems.


\subsection{Nomenclature}\label{Ssec:Nomencl}
Thoughout this article, $\R_{\ge 0} = [0,\infty)$ and $\|x\|$ is  the Euclidean norm of $x\in\R^n$. The symbols $\N$ denotes the set of natural numbers and $\N_0 = \N \cup\{0\}$. 
 The ring of real polynomials is denoted by $\R[s]$, and $\R(s)$ is its quotient field. In other words, $\R(s)$ is the field of real rational functions. Further, $ \Gl_n(\R)$
stands for the group of invertible matrices in $\R^{n\times n}$.

The  restriction of a function $f:V\to\R^n$ to $W\subseteq V$ is denoted by $\left.f\right|_{W}$, $V\subseteq W$. For $p\in[1,\infty]$,
$\cL^p(I\to\R^n)$ ($\cL^p_{\loc}(I\to\R^n)$) stands for the space of measurable and (locally) $p$-th power integrable functions $f:I\to\R^n$, $I\subseteq\R$ an interval. Likewise $\cL^\infty(I\to\R^n)$ ($\cL^\infty_{\loc}(I\to\R^n)$) is the space of measurable and (locally) essentially bounded functions $f:I\to\R^n$, and $\|f\|_\infty$ stands for the essential supremum of $f$. Note that functions which agree almost everywhere are identified. Further, for $p\in[1,\infty]$ and $k\in\N_0$,
$\cW^{k,p}(I\to\R^n)$ is the Sobolev space of elements of $L^{p}(I\to\R^n)$ ($L^{p}_{\loc}(I\to\R^n)$) with the property that the first $k$ weak derivatives exist and are elements of $L^{p}(I\to\R^n)$ ($L^{p}_{\loc}(I\to\R^n)$).
Moreover, $\cC^k(V\to\R^n)$ is the set of $k$-times continuously differentiable functions $f:V\to \R^n$, $V\subseteq\R^m$, and we set $\cC(V\to\R^n) := \cC^0(V\to\R^n)$.

\section{Linear systems and the truncated vector relative degree}\label{sec:genvrd}

In this section, we consider linear constant coefficient DAE systems
\begin{equation}\label{syst:DAE.line}\index{system!linear}
\begin{aligned}
\ddts E x(t) &= Ax(t)+Bu(t),\\
y(t)&=Cx(t),
\end{aligned}
\end{equation}
where $E,A\in \R^{l\times n}$, $B\in \R^{l\times m}$, $C\in \R^{p\times n}$. We denote the class of these systems by $\Sigma_{l,n,m,p}$ and write $[E,A,B,C]\in \Sigma_{l,n,m,p}$. We stress that these systems are not required to be {\it regular}, which would mean that $l=n$ and $\det(sE-A)\in\R[s]\setminus\{0\}$. The functions $u:\R \rightarrow \R^{m}$, $x:\R \rightarrow \R^n$, and $y:\R \rightarrow \R^p$ are called {\it input, (generalized) state variable}, and {\it output} of the system, respectively. We introduce the {\it behavior} of system~\eqref{syst:DAE.line} as\index{input}\index{output} \index{system!differential-algebraic}\index{system!regular}
\[\begin{aligned}
\fB_{[E,A,B,C]}:=&\left\{\phantom{\int}\!\!\!\!\!(x,u,y)\in\cL^1_{\loc}(\R\to\R^n\times \R^m\times \R^p)\ \right|\\
&\quad\left.\phantom{\int}\!\!\!\!\!\!\!\!\!\! Ex\in \cW^{1,1}_{\loc}(\R\to\R^l)\,\wedge\,\ddts E x= Ax+Bu\,\wedge\, y=Cx+Du\
\right\}.
\end{aligned}\]
For a~regular system $[E,A,B,C]\in \Sigma_{n,n,m,p}$, the {\em transfer function}\index{transfer function} is defined by
$$
G(s)=C(sE-A)^{-1}B\in \R(s)^{p\times m}.
$$

\subsection{Zero dynamics and right-invertibility}

To specify the class that we consider, we introduce the {\it zero dynamics} which are the set of solutions resulting in a trivial output. For more details on the concept of zero dynamics and a literature survey we refer to~\cite{Berg14a}.

\begin{definition}\leavevmode
\label{Def.zero.dyna}
The zero dynamics of $[E,A,B,C]\in \Sigma_{l,n,m,p}$ are the set\index{zero dynamics}
$$
\cZD_{[E,A,B,C]}:=\setdef{(x,u,y)\in \fB_{[E,A,B,C]}}{y= 0}.
$$
We call $\cZD_{[E,A,B,C]}$ {\em autonomous}, if\index{zero dynamics!autonomous}
$$
\forall\, \omega \in\cZD_{[E,A,B,C]}\ \forall\, I\subseteq \R \text{ open interval:}\ \ \omega|_I = 0\ \Rightarrow\ \omega= 0,
$$
and {\em asymptotically stable}, if\index{zero dynamics!asymptotically stable}
$$
\forall\, (x,u,y)\in \cZD_{[E,A,B,C]}: \lim\limits_{t\to \infty}\left\|\left.(x,u)\right|_{[t,\infty)}\right\|_\infty=0.
$$
\end{definition}

\begin{remark}\label{rem:zd}\leavevmode
Let $[E,A,B,C]\in \Sigma_{l,n,m,p}$.
\begin{enumerate}[a)]
\item\label{rem:zda} It has been shown in \cite[Prop.~3.5]{Berg16b} that
\[\text{$\cZD_{[E,A,B,C]}$ are autonomous}\ \
\Longleftrightarrow\ \ \ker_{\R(s)}\left[\begin{smallmatrix}-sE+A&B\\C&0\end{smallmatrix}\right]=\{0\}.
\]
In particular, $\left[\begin{smallmatrix}-sE+A&B\\C&0\end{smallmatrix}\right]$ is left invertible over $\R(s)$ if, and only if, $\cZD_{[E,A,B,C]}$ are autonomous. If $[E,A,B,C$ is regular, its transfer function $G(s)$ satisfies
\begin{equation}
\begin{bmatrix}-sE+A&B\\C&0\end{bmatrix}\begin{bmatrix}I_n&(sE-A)^{-1}B\\0&I_m\end{bmatrix}= \begin{bmatrix}-sE+A&0\\C&G(s)\end{bmatrix},\label{eq:tfschur}
\end{equation}
hence autonomy of the zero dynamics is equivalent to $G(s)$ having full column rank over $\R(s)$, cf.~\cite[Prop.~4.8]{Berg16b}.

\item\label{rem:zdb}
It has been shown in \cite[Lem.~3.11]{Berg16b} that
\[
\begin{aligned}
&\;\text{$\cZD_{[E,A,B,C]}$ are asymptotically stable}\\
\Longleftrightarrow\ &\; \ker_{\C}\left[\begin{smallmatrix}-\lambda E+A&B\\C&0\end{smallmatrix}\right]=\{0\}\;\text{ for all $\lambda\in\C_+$ with $\RE(\lambda)\geq0$}.
\end{aligned}\]
\end{enumerate}\end{remark}
We will consider systems with autonomous zero dynamics throughout this article. We will furthermore assume that the system is {\em right-invertible}, which is defined in the following.

\begin{definition}\label{Def.invertible.sys}\index{system!right-invertible}\index{right-invertible}
The system $[E,A,B,C]\in \Sigma_{l,n,m,p}$ is called
{\it right-invertible}, if
$$
\forall\, y\in \cC^{\infty}(\R\to\R^p)\ \exists\, (x,u)\in \cL^1_{\loc}(\R\to\R^n\times\R^m):\ (x,u,y)\in \fB_{[E,A,B,C]}.
$$
\end{definition}
The notion of right-invertibility has been used in \cite[Sec.~8.2]{TrenStoo01} for systems governed by ordinary differential equations and in \cite{Berg14c,Berg16b} for the differential-algebraic case. The concept is indeed motivated by tracking control: Namely, right-invertibility means that any smooth signal can be tracked by the output on a~right-invertible system.

\begin{remark}\label{rem:rightinv}
Consider a regular system  $[E,A,B,C]\in\Sigma_{n,n,m,p}$ with transfer function $G(s)$. It has been shown in \cite[Prop.~4.8]{Berg16b} that
\[
\text{$[E,A,B,C]$ is right-invertible}\ \
\Longleftrightarrow\ \ \im_{\R(s)}G(s)=\R(s)^{p},
\]
whence, by \eqref{eq:tfschur},
\[
\text{$[E,A,B,C]$ is right-invertible}\ \
\Longleftrightarrow\ \ \im_{\R(s)}\left[\begin{smallmatrix}-sE+A&B\\C&0\end{smallmatrix}\right]=\R(s)^{n+p},
\]
Combining this with Remark~\ref{rem:zd}~\ref{rem:zda}, we can infer from the dimension formula that for regular square systems\index{system!square} $[E,A,B,C]\in \Sigma_{n,n,m,m}$ (i.e., the dimensions of input and output coincide) with transfer function $G(s)\in\R(s)^{m\times m}$, the following statements are equivalent:
\begin{enumerate}[(i)]
\item  $\cZD_{[E,A,B,C]}$ autonomous,
\item $[E,A,B,C]$ is right-invertible,
\item $G(s)\in\R(s)^{m\times m}$ is invertible over $\R(s)$,
\item $\left[\begin{smallmatrix}-s E+A&B\\C&0\end{smallmatrix}\right]$ is invertible over $\R(s)$.
\end{enumerate}
\end{remark}

For general right-invertible systems with autonomous zero dynamics, we can derive a certain normal form under state space transformation. The following result is a straightforward combination of~\cite[Lem.~4.2~\&~Thm.~4.3~\&~Prop.~4.6]{Berg16b}.

\begin{theorem}\label{Thm:DAE.syst.inve}
Let a~right-invertible system $[E,A,B,C]\in \Sigma_{l,n,m,p}$ with autonomous zero dynamics be given. Then there exist $W\in \Gl_l(\R)$, $T\in \Gl_n(\R)$ such that
\begin{equation}\label{DAE.Thomas.form}
\begin{aligned}
W(sE-A)T=&\begin{bmatrix}
sI_{n_1}-Q&-A_{12}&0\\
-A_{21}&sE_{22}-A_{22}&sE_{23}\\
0&sE_{32}&sN-I_{n_3}\\0&0&-sE_{43}
\end{bmatrix},\;\; WB=\begin{bmatrix}
0\\
I_m\\
0\\0
\end{bmatrix},\\ CT=&\begin{bmatrix}0&I_p&0\end{bmatrix},
\end{aligned}
\end{equation}
where $n_1,n_3,n_4\in\N_0$, $N\in \R^{n_3\times n_3}$ is nilpotent \index{matrix!nilpotent} and
\[\begin{aligned}
Q\in\R^{n_1\times n_1},\quad A_{12}\in& \R^{n_1 \times p},& A_{21}\in& \R^{m \times n_1},\\
E_{22}, A_{22}\in& \R^{m \times p},& E_{23}\in&\R^{m\times n_3},\\
E_{32}\in &\R^{n_3 \times p},& E_{43}\in& \R^{n_4 \times n_3}
\end{aligned}\]
are such that $E_{43}N^jE_{32}=0$ for all $j\in \N_0$.
\end{theorem}

\begin{remark}\label{Rem:DAE.syst.inve}\leavevmode
Let $[E,A,B,C]\in \Sigma_{l,n,m,p}$ be right-invertible and have autonomous zero dynamics. Using the form~\eqref{DAE.Thomas.form}, we see that
$(x,u,y)\in \fB_{[E,A,B,C]}$ 
if, and only if,
\[Tx=(\eta^{\top},y^{\top},x_3^{\top})^{\top}\in \cL^{1}_{\loc}(\R\to\R^{n_1+p+n_3})\]
 satisfies
\begin{equation}
\begin{aligned}
\begin{bmatrix} E_{22}\\ E_{32}\end{bmatrix} y&\in\cW^{1,1}_{\loc}(\R\to\R^{m+n_3}), &\begin{bmatrix} E_{23}\\ N\\ E_{43}\end{bmatrix} x_3&\in \cW^{1,1}_{\loc}(\R\to\R^{m+n_3+n_4})
\end{aligned}\label{eq:smoothness}\end{equation}
and the equations
\begin{subequations}\label{eq:normalform}
\begin{align}
\dot{\eta}&=Q\eta+A_{12}y,\label{eq:normalform1}\\
0&=-\sum\limits_{i=0}^{\nu-1}E_{23}N^{i}E_{32}y^{(i+2)}-E_{22}\dot{y}+A_{22}y+A_{21}\eta+u,\label{eq:normalform2}\\
x_3&=\sum\limits_{i=0}^{\nu-1}N^{i}E_{32}y^{(i+1)}\label{eq:normalform3}
\end{align}
\end{subequations}
holds in the distributional sense.
In particular, the zero dynamics of $[E,A,B,C]$ are asymptotically stable if, and only if, any eigenvalue of $Q$ has negative real part.\\
Further note that $\eta\in\cL^{1}_{\loc}(\R\to\R^{n_1})$, $y\in\cL^{1}_{\loc}(\R\to\R^{p})$  together with \eqref{eq:normalform1} imply that $\eta\in\cW^{1,1}_{\loc}(\R\to\R^{n_1})$.
\end{remark}

\subsection{Truncated vector relative degree}\label{Sec:gen.vec.rel.deg}

Our aim in this section is to present a~suitable generalization of the concept of vector relative degree to differential-algebraic systems which are not necessarily regular. For regular systems a definition of this concept is given in~\cite[Def.~B.1]{Berg16b}.

\begin{definition}\label{def:DAE.line.vect.rel.deg}\leavevmode\index{vector relative degree}\index{relative degree!strict}\index{strict relative degree}
Let a~regular system $[E,A,B,C]\in \Sigma_{n,n,m,p}$  with transfer function $G(s)\in\R(s)^{p\times m}$ be given. We say that $[E,A,B,C]$ has {\it vector relative degree} $(r_1,\ldots, r_p)\in \Z^{1\times p}$, if there exists a matrix $\Gamma\in \R^{p\times m}$ with $\rk \Gamma=p$ and
$$
\lim\limits_{\lambda\to\infty}\diag(\lambda^{r_1},\dots, \lambda^{r_p})G(\lambda)=\Gamma.
$$
If the above holds with $r_1=\ldots=r_p=:r$, then we say that $[E,A,B,C]$ has {\it strict relative degree} $r$.
\end{definition}

Since this definition involves the transfer function, it is only applicable to regular systems. To avoid this limitation, we introduce a novel concept. Let us start by introducing the notion of column degree of a rational matrix. This generalizes the concept of column degree for polynomial matrices in~\cite[Sec.~2.4]{FuhrHelm15}.

\begin{definition}
For a rational function $r(s) = \frac{p(s)}{q(s)}\in\R(s)$ we define
\[
    \deg r(s):=\deg p(s) - \deg q(s).
\]
Further, for $r(s)=(r_1(s),r_2(s),\dots,r_p(s))^{\top}\in \R(s)^p$ we define
\[\deg r(s)=\max\limits_{1\le i\le p} \deg r_i(s).\]
\end{definition}

Note that the degree of a~rational function $r(s) = \frac{p(s)}{q(s)}$ is independent of the choice of~$p(s)$ and~$q(s)$, i.e., they do not need to be coprime.

If $[E,A,B,C]\in \Sigma_{l,n,m,p}$ has autonomous zero dynamics, then we can conclude from Remark~\ref{rem:zd} that $\left[\begin{smallmatrix}-sE+A&B\\C&0\end{smallmatrix}\right]\in\R(s)^{(l+p)\times(n+m)}$ possesses a~left inverse $L(s)\in \R(s)^{(n+m)\times(l+p)}$. Then we set
\begin{equation}\label{eq:Hs.def}
H(s):=-\begin{bmatrix}
0&I_m
\end{bmatrix}L(s)\begin{bmatrix}
0\\
I_p
\end{bmatrix}\in \R(s)^{m\times p}.
\end{equation}

\begin{remark}\label{rem:regsys}\leavevmode
\begin{enumerate}[a)]
\item\label{rem:regsysa} Assume that $[E,A,B,C]\in \Sigma_{l,n,m,p}$ has autonomous zero dynamics and is right-invertible. Then it has been shown in {\cite[Lem.~A.1]{Berg16b}}
that the rational matrix $H(s)\in \R(s)^{m\times p}$ is uniquely determined by $[E,A,B,C]$. Moreover, with the notation from Theorem~\ref{Thm:DAE.syst.inve}, we have
\begin{equation}H(s)=sE_{22}-A_{22}-A_{21}(sI_{n_1}-Q)^{-1}A_{12}-s^2E_{23}(sN-I_{n_3})^{-1}E_{32}.\label{eq:H(s)repr}\end{equation}
We stress that the above representation is independent of the transformation matrices~$W$ and~$T$ in~\eqref{DAE.Thomas.form}.

\item\label{rem:regsysb} If $[E,A,B,C]\in \Sigma_{n,n,m,m}$ has autonomous zero dynamics and is regular with transfer function $G(s)\in\R(s)^{m\times m}$, then, invoking~\eqref{eq:tfschur} and Remark~\ref{rem:rightinv}, it can be shown that $H(s)=G(s)^{-1}$, see also \cite[Rem.~A.4]{Berg16b}.
\end{enumerate}
\end{remark}

In view of Remark~\ref{rem:regsys}, we see that for any regular system $[E,A,B,C]\in \Sigma_{n,n,m,m}$ with transfer function $G(s)$ and vector relative degree $(r_1,\dots,r_m)$, we have
\begin{equation}
\begin{aligned}
&\lim\limits_{\lambda\to\infty}\diag(\lambda^{r_1},\dots, \lambda^{r_m})G(\lambda)=\Gamma\in \Gl_m(\R)\\
\Longleftrightarrow \quad &\lim\limits_{\lambda\to\infty}H(\lambda)\diag(\lambda^{-r_1},\dots, \lambda^{-r_m})=\Gamma^{-1}\in \Gl_m(\R),\end{aligned}\label{eq:invreldeg}
\end{equation}
with $H(s)$ as in \eqref{eq:Hs.def}. This motivates to use $H(s)$ instead of the transfer function $G(s)$ to define a generalization of the vector relative degree to DAE systems which are not necessarily regular.

\begin{definition}\label{def:DAE.line.gene.vect.rel.deg}
Assume that $[E,A,B,C]\in \Sigma_{l,n,m,p}$ is right-invertible and has autonomous zero dynamics. Let $H(s)\in\R(s)^{m\times p}$ be defined as in \eqref{eq:Hs.def} (which is well-defined by Remark~\ref{rem:regsys}~\ref{rem:regsysa}), $h_i(s) = H(s)e_i\in \R(s)^{m}$ for $i=1,\ldots,p$ and set $r_i=\max\{\deg h_i(s),0\}$. Let~$q$ be the number of nonzero entries of $(r_1,\ldots,r_p)$,
\begin{equation}\label{eq:gen.vec.rel.deg.def}
\hat{\Gamma}:=\lim\limits_{\lambda\to \infty}H(\lambda)\diag(\lambda^{-r_1},\dots, \lambda^{-r_p})\in \R^{m\times p},
\end{equation}
and  $\hat{\Gamma}_q\in \R^{m\times q}$ be the matrix which is obtained from $\hat{\Gamma}$ by deleting all the columns corresponding to $r_i=0$. Then we call $r=(r_1,\dots, r_p)\in \N_0^{1\times p}$ the {\it truncated vector relative degree} of the system $[E,A,B,C]$, if $\rk \hat{\Gamma}_q=q$.\\
A~truncated vector relative degree $(r_1,\dots, r_p)$ is called {\em ordered}, if $r_1\ge \ldots \ge r_p$.
\end{definition}\index{vector relative degree!truncated}\index{truncated vector relative degree}\index{vector relative degree!truncated ordered}\index{truncated vector relative degree!ordered}

\begin{remark}\label{rem:genvrd}\leavevmode
Let the system $[E,A,B,C]\in\Sigma_{l,n,m,p}$  be right invertible and have autonomous zero dynamics.
\begin{enumerate}[a)]
\item\label{rem:genvrda} Assume that $[E,A,B,C]$ has ordered truncated vector relative degree $(r_1,\dots,r_q,0,\ldots,0)$ with $r_q>0$. Then the matrices $\hat{\Gamma}$ and $\hat{\Gamma}_q$ in Definition~\ref{def:DAE.line.gene.vect.rel.deg} are related by
\[
{\hat{\Gamma}}_q={\hat{\Gamma}} \begin{bmatrix}
I_q\\
0
\end{bmatrix}.\]
\item\label{rem:genvrdb} Assume that $[E,A,B,C]$ has truncated vector relative degree $(r_1,\dots,r_p)\in \N_0^{1\times p}$. Consider a~permutation matrix $P_\sigma\in \R^{p\times p}$ induced by the permutation $\sigma:\{1,\ldots, p\}\to \{1,\ldots, p\}$. A straightforward calculation shows that $H_\sigma(s)$ as in~\eqref{eq:Hs.def} corresponding to $[E,A,B,P_\sigma C]$ satisfies $H_\sigma(s) = H(s) P_\sigma$, thus the system $[E,A,B,P_\sigma C]$ has truncated vector relative degree $(r_{\sigma(1)},\dots,r_{\sigma(p)})$.\index{permutation}\index{permutation!ordered}
    In particular, there exists a permutation $\sigma$ such that the output-permuted system $[E,A,B,P_\sigma C]$ has ordered truncated vector relative degree.

\item\label{rem:genvrdd}
Assume that $[E,A,B,C]$ has ordered truncated vector relative degree
$(r_1,\dots, r_p)\in \N_0^{1\times p}$. Using the notation from Theorem~\ref{Thm:DAE.syst.inve} and~\eqref{eq:H(s)repr}, we obtain that
\[
\begin{aligned}
\hat{\Gamma}&=\lim\limits_{\lambda\to \infty}H(\lambda)\diag(\lambda^{-r_1},\dots, \lambda^{-r_p})\\
&=\lim\limits_{\lambda\to \infty}\big[(\lambda E_{22}-A_{22})-A_{21}(\lambda I-Q)^{-1}A_{12}-\lambda^2E_{23}(\lambda N-I_{n_3})^{-1}E_{32}\big]\cdot\\&\qquad\quad\;\cdot\diag(\lambda^{-r_1},\dots, \lambda^{-r_p})\\
&=\lim\limits_{\lambda\to \infty}\left[\lambda E_{22}-A_{22}+\sum\limits_{k=0}^{\nu-1} \lambda^{k+2}E_{23}N^kE_{32}\right]\diag(\lambda^{-r_1},\dots, \lambda^{-r_p}).
\end{aligned}
\]

\item\label{rem:genvrde}
Consider a regular system $[E,A,B,C]\in \Sigma_{n,n,m,p}$. If $m>p$, then, in view of Remark~\ref{rem:zd}~\ref{rem:zda}, the zero dynamics of $[E,A,B,C]$ are not autonomous, because $\begin{smallbmatrix} -sE+A& B\\ C&0\end{smallbmatrix}$ has a non-trivial kernel over $\R(s)$. Therefore, such a system does not have a truncated vector relative degree, but a vector relative degree may exist. As an example consider the system $[E,A,B,C]\in \Sigma_{1,1,2,1}$ with $E=C=[1]$, $A=[0]$ and $B=[1,1]$, for which a truncated relative degree does not exist. However, the transfer function is given by $G(s) = s^{-1} [1, 1]$ and hence the system even has strict relative degree $r=1$.

If $m\le p$ and $[E,A,B,C]$ has a vector relative degree, then also a truncated vector relative degree exists. This can be seen as follows: First observe that, as a consequence of Definition~\ref{def:DAE.line.vect.rel.deg}, $p\le m$ and hence we have $p=m$. Therefore, the matrix~$\Gamma\in\R^{m\times m}$ in Definition~\ref{def:DAE.line.vect.rel.deg} is invertible. Let $F(s) := \diag(\lambda^{r_1},\dots, \lambda^{r_m}) G(s)$, then $F(s) = \Gamma + G_{\rm sp}(s)$ for some {\it strictly proper}\index{transfer function!strictly proper} $G_{\rm sp}(s)\in\R(s)^{m\times m}$, i.e., $\lim_{\lambda\to\infty} G_{\rm sp}(\lambda) = 0$. Then $\tilde G(s) := -\Gamma^{-1} G_{\rm sp}(s)$ is strictly proper as well. Let $p(s)\in\R(s)^m$ be such that $F(s) p(s) =0$, then $p(s) = \tilde G(s) p(s)$. A component-wise comparison of the degrees yields that
\[
   \forall\, i=1,\ldots,m:\ \deg p_i(s) = \deg \sum_{j=1}^m \tilde G_{ij}(s) p_j(s) \le \max_{j=1,\ldots,m} \big(\deg p_j(s) -1\big),
\]
because $\deg \tilde G_{ij}(s) \le -1$ for all $i,j=1,\ldots,m$. Therefore,
\[
    \max_{i=1,\ldots,m} \deg p_i(s) \le \max_{j=1,\ldots,m} \big(\deg p_j(s) -1\big) = \left(\max_{j=1,\ldots,m} \deg p_j(s)\right) - 1,
\]
a contradiction. This shows that $F(s)$ is invertible over $\R(s)$ and hence $G(s)$ is invertible over $\R(s)$. Then Remark~\ref{rem:rightinv} yields that $[E,A,B,C]$ is right-invertible and has autonomous zero dynamics. Moreover, Remark~\ref{rem:regsys}~\ref{rem:regsysb} gives that $H(s)=-G(s)^{-1}$ and hence it follows that a truncated vector relative degree exists with $\hat\Gamma = -\Gamma^{-1}$ as in~\eqref{eq:gen.vec.rel.deg.def}.

\item\label{Rem:gen.vec.rel.degb} If $[E,A,B,C]\in \Sigma_{n,n,m,m}$ is regular and has autonomous zero dynamics, then $[E,A,B,C]$ has truncated vector relative degree $(0,\dots,0)$ if, any only if, the transfer function $G(s)\in\R(s)^{m\times m}$ of $[E,A,B,C]$ is {\it proper}, i.e.,\index{transfer function!proper} $\lim_{\lambda\to\infty}G(\lambda)\in\R^{m\times m}$ exists. This is an immediate consequence of the fact that, by Remark~\ref{rem:regsys}~\ref{rem:regsysb}, the matrix $H(s)$ in~\eqref{eq:Hs.def} satisfies $G(s)^{-1}$.

\item
A motivation for the definition of the truncated vector relative degree, even when only regular systems are considered, is given by output feedback control: Whilst the regular system $[E,A,B,C]\in \Sigma_{2,2,1,1}$ with
\[E=\begin{bmatrix}0&1\\0&0\end{bmatrix},\;\;A=\begin{bmatrix}1&0\\0&1\end{bmatrix},\;\; B=\begin{bmatrix}0\\1\end{bmatrix},\;\; C=\begin{bmatrix}1&0\end{bmatrix}\]
has transfer function $G(s)=-s$ and thus vector relative degree $(r_1)=(-1)$, application of the static output feedback $u(t)=Ky(t)+v(t)$ with new input $v$ leads to the system $[E,A+BKC,B,C]$ with transfer function $G_K(s)=\frac{-K}{1+K s}$. We may infer that the vector relative degree of $[E,A+BKC,B,C]$ is zero unless $K=0$, thus the vector relative degree is not invariant under output feedback in general.
\end{enumerate}\end{remark}

In the following we show that the truncated vector relative degree is however invariant under static output feedback.\index{output feedback!static}\index{output feedback}

\begin{proposition}
Let $[E,A,B,C]\in \Sigma_{l,n,m,p}$ and $K\in \R^{m\times p}$ be given. Then the following statements hold:
\begin{enumerate}[a)]
\item $\cZD_{[E,A,B,C]}$ are autonomous if, and only if, $\cZD_{[E,A+BKC,B,C]}$ are autonomous.
\item $[E,A,B,C]$ is right-invertible if, and only if, $[E,A+BKC,B,C]$ is right-invertible.
\item $[E,A,B,C]$ has a truncated vector relative degree if, and only if,  $[E,A+BKC,B,C]$  has a truncated vector relative degree. In this case, the truncated vector relative degrees of $[E,A,B,C]$ and $[E,A+BKC,B,C]$ coincide.
\end{enumerate}\end{proposition}
\begin{proof}
\begin{enumerate}[a)]
\item This follows from Remark~\ref{rem:zd}~\ref{rem:zda} together with
\begin{equation}
\begin{bmatrix}
-sE+A+BKC&B\\
C&0
\end{bmatrix}=\begin{bmatrix}
I_l&BK\\
0&I_p
\end{bmatrix}
\begin{bmatrix}
-sE+A&B\\
C&0
\end{bmatrix}.
\label{eq:feedbackmatr}
\end{equation}
\item Since $[E,A+BKC,B,C]$ is obtained from $[E,A,B,C]$ by output feedback $u(t)=Ky(t)+v(t)$ with new input $v\in\cL^1_{\loc}(\R\to\R^m)$, we obtain that $(x,u,y)\in\fB_{[E,A,B,C]}$ if, and only if, $(x,u-Ky,y)\in\fB_{[E,A+BKC,B,C]}$. In particular, the set of generated outputs of $[E,A,B,C]$ and $[E,A+BKC,B,C]$ are the same, whence $[E,A,B,C]$ is right-invertible if, and only if, $[E,A+BKC,B,C]$ is right-invertible.
\item Since $[E,A,B,C]$ is obtained from  $[E,A+BKC,B,C]$ by applying the feedback $-K$, it suffices to prove one implication. In view of Remark~\ref{rem:genvrd}~\ref{rem:genvrdb}, it is no loss of generality to assume that $[E,A,B,C]$ has ordered truncated vector relative degree $(r_1,\dots, r_q,0\ldots, 0)\in \N_0^{1\times p}$ with $r_q>0$.  Let
    $L(s), L_K(s)\in \R(s)^{(n+m)\times (l+p)}$ be left inverses of
    \[\begin{bmatrix}
-sE+A&B\\
C&0
\end{bmatrix}\text{ and }\begin{bmatrix}
-sE+A+BKC&B\\
C&0
\end{bmatrix},\quad \text{resp.,}\]
and partition
\[L(s)=\begin{bmatrix}L_{11}(s)&L_{12}(s)\\L_{21}(s)&H(s)\end{bmatrix}.\]
From~\eqref{eq:feedbackmatr} it follows that $L(s)\begin{smallbmatrix}
I_l&-BK\\
0&I_p
\end{smallbmatrix}$ is a left inverse of $\begin{smallbmatrix}
-sE+A+BKC&B\\
C&0
\end{smallbmatrix}$. Since $H_K(s) = [0, I_m] L_K(s) \begin{smallbmatrix} 0\\ I_p\end{smallbmatrix}$ is independent of the choice of the left inverse $L_K(s)$ by Remark~\ref{rem:regsys}~\ref{rem:regsysa}, we may infer that
$$
\begin{aligned}
H_K(s)&=\begin{bmatrix}
0&I_m
\end{bmatrix}L(s)\begin{bmatrix}
I_n&-BK\\
0&I_m
\end{bmatrix}\begin{bmatrix}
0\\
I_p
\end{bmatrix}\\
&=\begin{bmatrix}
0&I_m
\end{bmatrix}\begin{bmatrix}L_{11}(s)&L_{12}(s)\\L_{21}(s)&H(s)\end{bmatrix}\begin{bmatrix}
-BK\\
I_p
\end{bmatrix}\\
&=H(s)-L_{21}(s)BK.\\
\end{aligned}
$$
The relation $L(s)\left[\begin{smallmatrix}
-sE+A&B\\
C&0
\end{smallmatrix}\right]=I_{n+m}$ leads to $L_{21}(s)B=I_m$.
Therefore, $H_K(s) = H(s) - K$ and we find
$$\begin{aligned}
\hat\Gamma_K&=\lim\limits_{\lambda\to \infty}H_K(\lambda)\diag(\lambda^{-r_1},\dots, \lambda^{-r_q}, 1, \dots, 1)\\
&=\hat\Gamma-\lim\limits_{\lambda\to \infty} K\diag(\lambda^{-r_1},\dots, \lambda^{-r_q}, 1, \dots, 1).
\end{aligned}$$
This implies that $\hat{\Gamma}_K\left[\begin{smallmatrix}
I_q\\
0
\end{smallmatrix}\right]=\hat{\Gamma}_q$, and thus \[\rk \hat{\Gamma}_K\left[\begin{smallmatrix}
I_q\\
0
\end{smallmatrix}\right]=\rk\hat{\Gamma}_q=q.\] Therefore, the truncated vector relative degree of the feedback system $[E,A+BKC,B,C]$ is $(r_1,\dots, r_q,0\dots, 0)$, i.e., that of $[E,A,B,C]$.\qedhere
\end{enumerate}
\end{proof}

\begin{remark}\label{Rem:gen.vec.rel.deg}\leavevmode
\begin{enumerate}[a)]
\item\label{Rem:gen.vec.rel.dega}  The truncated vector relative degree of a right-invertible system with autonomous zero dynamics does not necessarily exist: For instance, consider $[E,A,B,C]\in\Sigma_{4,4,2,2}$ with
$$
E=\begin{bmatrix}
1&0&0&0\\
0&1&1&0\\
0&1&1&0\\
0&0&0&0
\end{bmatrix},\;\; A=\begin{bmatrix}
-1&0&0&0\\
0&1&-1&0\\
0&1&2&0\\
0&0&0&1
\end{bmatrix},\;\; B=\begin{bmatrix}
0&0\\
1&0\\
0&1\\
0&0
\end{bmatrix},\;\; C=\begin{bmatrix}
0&1&0&0\\
0&0&1&0
\end{bmatrix}.
$$
For this system, we have
$$
H(s)=\begin{bmatrix}
s-1&s+1\\
s-1&s-2
\end{bmatrix}.
$$
Moreover,
$$
\hat{\Gamma}=\lim\limits_{\lambda\to \infty}H(\lambda)\diag(\lambda^{-1},\lambda^{-1})=\begin{bmatrix}
1&1\\
1&1
\end{bmatrix} = \hat{\Gamma}_q.
$$
Since $\rk \hat{\Gamma}_q=1<2$, which is the number of columns of $H(s)$ with positive degree. Hence, this system does not have a truncated vector relative degree.

\item\label{Rem:gen.vec.rel.degc} There exist right-invertible regular systems with autonomous zero dynamics with the property that the truncated vector relative degree exists, but the
vector relative degree according to Definition~\ref{def:DAE.line.vect.rel.deg} does not exist. For instance, consider $[E,A,B,C]\in\Sigma_{5,5,2,2}$ with
\begin{equation}
E=\begin{bmatrix}
1&0&0&0&0\\
0&1&0&1&0\\
0&-1&0&0&0\\
0&0&0&0&1\\
0&1&0&0&0\\
\end{bmatrix},\;\; A=\begin{bmatrix}
-1&1&-2&0&0\\
3&5&0&0&0\\
0&0&0&0&0\\
0&0&0&1&0\\
0&0&0&0&1\\
\end{bmatrix},\;\;
B=\begin{bmatrix}
0&0\\
1&0\\
0&1\\
0&0\\
0&0\\
\end{bmatrix},\;\; C=\begin{bmatrix}
0&1&0&0&0\\
0&0&1&0&0\\
\end{bmatrix}.
\label{eq:exlin}
\end{equation}
Then
$$
G(s)=C(sE-A)^{-1}B=\begin{bmatrix}
0&-\tfrac{1}{s}\\
\tfrac{s+1}{6}&\tfrac{s^4+s^3+s^2-4s-8}{6s}
\end{bmatrix}.
$$
We have
$$
\Gamma:=\lim\limits_{\lambda\to \infty}\diag(\lambda,\lambda^{-3})G(\lambda)=\begin{bmatrix}
0&-1\\
0&\tfrac{1}{6}
\end{bmatrix}, \text{ and } \rk\Gamma=1<2.
$$
This implies that the system does not have vector relative degree in the sense of Definition~\ref{def:DAE.line.vect.rel.deg}. Invoking Remark~\ref{rem:regsys}~\ref{rem:regsysb}, we obtain
$$
H(s)=G(s)^{-1}=\begin{bmatrix}
\tfrac{s^4+s^3+s^2-4s-8}{s+1}&\tfrac{6}{s+1}\\
-s&0
\end{bmatrix},
$$
and
$$
\hat{\Gamma}:=\lim\limits_{\lambda\to \infty}H(\lambda)\diag(\lambda^{-3},1)=\begin{bmatrix}
1&0\\
0&0
\end{bmatrix}\; \text{ and } \hat{\Gamma}_q=\begin{bmatrix}
1\\
0
\end{bmatrix}.
$$
Then $\rk\hat{\Gamma}_q=1 = q$, and consequently this system has truncated vector relative degree $(3,0)$.
\end{enumerate}
\end{remark}

\subsection{A~representation for systems with truncated vector relative degree}\label{Sec:norm.form}

For ODE systems, {\sc Byrnes} and {\sc Isidori} have introduced a~normal
form under state space transformation which allows to read off the relative degree and internal dynamics\index{internal dynamics} \cite{ByrnIsid84,Isid95}. This normal form plays an important role in designing local and global stabilizing feedback controllers for nonlinear systems~\cite{ByrnIsid85,ByrnIsid88,ByrnIsid89}, adaptive observers \cite{NicoTorn89},\index{observer!adaptive} and adaptive controllers\index{controller!adaptive}\index{adaptive!controller}\index{adaptive!observer} \cite{IlchRyan94,IlchTown93}. A~normal form for linear ODE systems with vector relative degree has been developed in~\cite{Muel09a}. Further, a~normal form for regular linear DAE systems with strict relative degree has been derived in~\cite{BergIlch12b}, whereas
a~normal form for regular linear differential-algebraic systems with proper inverse transfer function in~\cite{BergIlch12a}. The latter has been extended to (not necessarily regular) DAE systems with truncated vector relative degree pointwise less or equal to one in~\cite{Berg16b}, although this notion was not used there. Note that the concept of truncated vector relative degree encompasses systems governed by ODEs with strict or vector relative degree as well as regular DAE systems with strict relative degree (up to some extent, cf.\ Remark~\ref{rem:genvrd}~\ref{rem:genvrde}) or proper inverse transfer function, and we introduce a~novel representation which comprises
all the aforementioned results.

Assume that $[E,A,B,C]\in\Sigma_{l,n,m,p}$ is right-invertible, has autonomous zero dynamics and has possesses a~truncated vector relative degree $(r_1,\ldots,r_p)\in\N_0^{1\times p}$. By Remark~\ref{rem:genvrd}~\ref{rem:genvrdb}, it is further no loss of generality to assume that the latter is ordered, i.e., $r_1\ge\ldots\ge r_q>0 = r_{q+1} = \ldots = r_p$. Introduce the polynomial matrix
\[
    F(s) :=  sE_{22}-A_{22}+\sum\limits_{k=0}^{\nu-1} s^{k+2}E_{23}N^kE_{32} \in\R(s)^{m\times p}.
\]
By Remark~\ref{rem:genvrd}~\ref{rem:genvrdd} we have
\begin{align*}
\hat{\Gamma}&=\lim\limits_{\lambda\to \infty}H(\lambda)\diag(\lambda^{-r_1},\dots, \lambda^{-r_q}, 1, \dots, 1)=\lim\limits_{\lambda\to \infty} F(\lambda) \diag(\lambda^{-r_1},\dots, \lambda^{-r_q}, 1, \dots, 1)\\
&=\begin{bmatrix}
\hat{\Gamma}_{11}&\hat{\Gamma}_{12}\\
\hat{\Gamma}_{21}&\hat{\Gamma}_{22}\\
\end{bmatrix} \in \R^{m\times p},
\end{align*}
where the latter partition is with $\hat{\Gamma}_{11}\in \R^{q\times q}$, $\hat{\Gamma}_{12}\in \R^{q\times (p-q)}$, $\hat{\Gamma}_{21}\in \R^{(m-q)\times q}$ and $\hat{\Gamma}_{22}\in \R^{(m-q)\times (p-q)}$.
Then Definition~\ref{def:DAE.line.gene.vect.rel.deg} yields
$$
\rk\begin{bmatrix}
\hat{\Gamma}_{11}\\
\hat{\Gamma}_{21}
\end{bmatrix}=\rk\hat{\Gamma}\begin{bmatrix}
I_q\\
0
\end{bmatrix}=q.
$$
Let $h\in\N$ be such that $r_h>1$ and $r_{h+1}=1$. Denote the $j$th column of a~matrix $M$ by $M^{(j)}$. Then
\begin{align*}
\hat{\Gamma}
&=\lim\limits_{\lambda\to \infty} F(\lambda) \diag(\lambda^{-r_1},\dots, \lambda^{-r_q}, 1, \dots, 1)\\
&=\begin{bmatrix}
E_{23}N^{r_1-2}E^{(1)}_{32} & \ldots & E_{23}N^{r_h-2}E^{(h)}_{32} & E^{(h+1)}_{22}&\ldots&E^{(q)}_{22}&-A^{(q+1)}_{22}&\ldots&-A^{(p)}_{22}
\end{bmatrix},
\end{align*}
and thus
\begin{equation}\label{eq:DAE.line.Gamma.p}
\hat{\Gamma}_q=\begin{bmatrix}
\hat{\Gamma}_{11}\\
\hat{\Gamma}_{21}
\end{bmatrix}=\begin{bmatrix}
E_{23}N^{r_1-2}E^{(1)}_{32} & \ldots & E_{23}N^{r_h-2}E^{(h)}_{32} & E^{(h+1)}_{22}&\ldots&E^{(q)}_{22}
\end{bmatrix}\in \R^{m\times q}.
\end{equation}
Since $\rk\hat{\Gamma}_q=q$, by reordering the inputs\index{input!renumbering} and -- accordingly -- reording the rows of $A_{21}$, $E_{22}$, $A_{22}$ and $E_{23}$, it is no loss of generality to assume that the first $q$ rows of $\hat{\Gamma}_q$ are linearly independent, thus $\hat{\Gamma}_{11}\in \Gl_q(\R)$.
Consider the matrix
\begin{equation}\label{eq:DAE.line.Gamma.form}
\Gamma:=\begin{bmatrix}
\Gamma_{11}&0\\
\Gamma_{21}&I_{m-q}\\
\end{bmatrix}\in \Gl_m(\R),
\end{equation}
where $\Gamma_{11}=\hat{\Gamma}^{-1}_{11}\in \Gl_q(\R)$, $\Gamma_{21}=-\hat{\Gamma}_{21}\hat{\Gamma}^{-1}_{11}\in \R^{(m-q)\times q}$, then
\begin{equation}\label{eq:DAE.line.Gamma.form2}
\Gamma\hat{\Gamma}_q=\begin{bmatrix}
I_q\\
0\\
\end{bmatrix}.
\end{equation}
On the other hand, using the notation from Theorem~\ref{Thm:DAE.syst.inve} and invoking Remark~\ref{Rem:DAE.syst.inve}, we have that
$(x,u,y)\in \fB_{[E,A,B,C]}$ if, and only if,
$Tx=(\eta^{\top},y^{\top},x_3^{\top})^{\top}\in \cL^{1}_{\loc}(\R\to\R^{n_1+p+n_3})$
solves~\eqref{eq:normalform} in the distributional sense, and the components satisfy~\eqref{eq:smoothness}. Since~\eqref{eq:normalform2} can be written as $F(\ddts)y = A_{21}\eta + u$, by construction of $\hat \Gamma_q$ and~\eqref{eq:DAE.line.Gamma.p} we may rewrite this as
\begin{equation}
\begin{aligned}\hat{\Gamma}_q\begin{pmatrix}y_1^{(r_1)}\\\vdots\\y_q^{(r_q)}\end{pmatrix}=&\,M_1\begin{pmatrix}y_1\\\vdots\\y_1^{(r_1-1)}\end{pmatrix}+\ldots+M_q
\begin{pmatrix}y_q\\\vdots\\y_q^{(r_q-1)}\end{pmatrix}+M\begin{pmatrix}y_{q+1}\\\vdots\\y_m\end{pmatrix}+A_{21}\eta+u\end{aligned}\label{eq:secondeq}
\end{equation}
for some $M_1\in\R^{m\times r_1},\ldots,M_q\in\R^{m\times r_q}$, $M\in\R^{m\times (p-q)}$ which can be constructed from the columns of
$E_{23}N^{i}E_{32}$, $E_{22}$ and $A_{22}$, $i=0,\ldots,r_1$. Define $R_{j,1}\in\R^{q\times r_j}$, $R_{j,2}\in\R^{(m-q)\times r_j}$ for $j=1,\ldots,q$ and $S_1\in\R^{q\times (p-q)}$, $S_2\in\R^{(m-q)\times (p-q)}$, $P_1\in\R^{q\times n_1}$, $P_2\in\R^{(m-q)\times n_1}$
by
\begin{equation}\begin{bmatrix}R_{j,1}\\R_{j,2}\end{bmatrix}:=\Gamma M_j,\quad j=1,\ldots,q,\qquad \begin{bmatrix}S_1\\S_2\end{bmatrix}:=\Gamma M,\qquad \begin{bmatrix}P_1\\P_2\end{bmatrix}:=\Gamma A_{21}.\label{eq:RSdef}\end{equation}
%
By a~multiplication of \eqref{eq:secondeq} from the left with $\Gamma\in\Gl_{m}(\R)$, we obtain that, also invoking~\eqref{eq:normalform1} and~\eqref{eq:normalform3},
\begin{equation}
\begin{aligned}
\dot{\eta}&=Q\eta+A_{12}y,\\
\begin{pmatrix}y_1^{(r_1)}\\\vdots\\y_q^{(r_q)}\end{pmatrix}&=R_{1,1}\begin{pmatrix}y_1\\\vdots\\y_1^{(r_1-1)}\end{pmatrix}\!+\ldots+R_{q,1}
\begin{pmatrix}y_q\\\vdots\\y_q^{(r_q-1)}\end{pmatrix}\!+S_1\begin{pmatrix}y_{q+1}\\\vdots\\y_m\end{pmatrix}+
P_1\eta+\Gamma_{11}\begin{pmatrix}
u_1\\
\vdots\\
u_q
\end{pmatrix},\\
0&=R_{1,2}\begin{pmatrix}y_1\\\vdots\\y_1^{(r_1-1)}\end{pmatrix}+\ldots+R_{q,2}
\begin{pmatrix}y_q\\\vdots\\y_q^{(r_q-1)}\end{pmatrix}+S_2\begin{pmatrix}y_{q+1}\\\vdots\\y_m\end{pmatrix}\\&\quad+
P_2 \eta+\Gamma_{21}\begin{pmatrix}
u_1\\
\vdots\\
u_q
\end{pmatrix}+\begin{pmatrix}
u_{q+1}\\
\vdots\\
u_m
\end{pmatrix},\\
x_3&=\sum\limits_{i=0}^{\nu-1}N^{i}E_{32}y^{(i+1)}.
\label{DAE.line.norm.form}\end{aligned}\end{equation}

We have thus derived a representation for systems with truncated vector relative degree and summarize the findings in the following result.

\begin{theorem}\label{Thm:DAE.line.norm.form}
Let a~right-invertible system $[E,A,B,C]\in \Sigma_{l,n,m,p}$ with autonomous zero dynamics be given. Assume that $[E,A,B,C]$ has ordered truncated vector relative degree $(r_1,\dots,r_q,0,\dots,0)$ with $r_q>0$. Use the notation from Theorem~\ref{Thm:DAE.syst.inve},~\eqref{eq:DAE.line.Gamma.form} and~\eqref{eq:RSdef}.
Then $(x,u,y)\in \fB_{[E,A,B,C]}$, 
if, and only if,  after a reordering of the inputs so that $\hat\Gamma_{11}$ in~\eqref{eq:DAE.line.Gamma.p} is invertible,
\[Tx=(\eta^{\top},y^{\top},x_3^{\top})^{\top}\in \cL^{1}_{\loc}(\R\to\R^{n_1+p+n_3})\]
satisfies the smoothness conditions in \eqref{eq:smoothness} and solves~\eqref{DAE.line.norm.form} in the distributional sense.
\end{theorem}

\begin{remark}\label{Rems:DAE.normal.form.special.cases} \leavevmode
Consider a~regular and right-invertible system $[E,A,B,C]\in \Sigma_{n,n,m,m}$ with autonomous zero dynamics and ordered truncated vector relative degree $(r_1,\dots,r_q,0,\dots,0)\in \N_0^{1\times p}$ such that $r_q>0$.
\begin{enumerate}[a)]
\item\label{Rems:DAE.normal.form.special.casesb}
If $[E,A,B,C]$ has strict relative degree $r>0$, then $q=m$ and $r_1=\ldots=r_m=r$. In this case, the representation~\eqref{DAE.line.norm.form} simplifies to
\[\begin{aligned}
\dot{\eta}&=Q\eta+A_{12}y,\\
y^{(r)}&=R_{1,1}\begin{pmatrix}y_1\\\vdots\\y_1^{(r-1)}\end{pmatrix}+\ldots+R_{m,1}
\begin{pmatrix}y_m\\\vdots\\y_m^{(r-1)}\end{pmatrix}+
P_1 \eta+\Gamma_{11}u,\\
x_3&=\sum\limits_{i=0}^{\nu-1}N^{i}E_{32}y^{(i+1)}.
\end{aligned}
\]
Since the second equation can be rewritten as
\[y^{(r)}=Q_{r-1}y^{(r-1)}+\ldots+Q_0y+P_1\eta+\Gamma_{11}u\]
for matrices $Q_0,\ldots,Q_{r-1}$, this is exactly the form which has been developed in~\cite{BergIlch12b}.

\item If the transfer function $G(s)\in\R(s)^{m\times m}$ of $[E,A,B,C]$ has a~proper inverse,\index{transfer function!properly invertible} then we have that $H(s)=G(s)^{-1}$ (see Remark~\ref{rem:regsys}~\ref{rem:regsysb}) is proper, hence $q=0$ and the truncated vector relative degree is $(0,\ldots,0)$. In this case, the representation~\eqref{DAE.line.norm.form} simplifies to
\[
\begin{aligned}
\dot{\eta}&=Q\eta+A_{12}y,\\
0&=S_2y+P_2\eta+u,\\
x_3&=\sum\limits_{i=0}^{\nu-1}N^{i}E_{32}y^{(i+1)},
\end{aligned}
\]
which is exactly the form developed in \cite{BergIlch12a}.

\item If the system is an ODE, that is $E=I_n$, then its transfer function~$G(s)$ is strictly proper,\index{transfer function!strictly proper} i.e., $\lim_{\lambda\to\infty}G(\lambda)=0$. We can further infer from Remark~\ref{rem:rightinv} that the transfer function $G(s)\in\R(s)^{m\times m}$ is invertible. Then~\eqref{eq:invreldeg} implies $q=m$, i.e., the truncated vector relative degree (which coincides with the vector relative degree by Remark~\ref{rem:genvrd}~\ref{rem:genvrde}) is $(r_1,\ldots,r_m)\in\N^{1\times m}$.
     In this case,~\eqref{DAE.line.norm.form} simplifies to
\[\begin{aligned}
\dot{\eta}&=Q\eta+A_{12}y,\\
\begin{pmatrix}y_1^{(r_1)}\\\vdots\\y_m^{(r_m)}\end{pmatrix}&=R_{1,1}\begin{pmatrix}y_1\\\vdots\\y_1^{(r_1-1)}\end{pmatrix}+\ldots+R_{m,1}
\begin{pmatrix}y_m\\\vdots\\y_m^{(r_m-1)}\end{pmatrix}+ P_1\eta+\Gamma_{11}u,\\
x_3&=\sum\limits_{i=0}^{\nu-1}N^{i}E_{32}y^{(i+1)}.
\end{aligned}
\]
This form comprises the one presented in \cite{Muel09a}, where, additionally,
\begin{align*}
x_3&=(\dot{y}_1,\dots,y^{(r_1-1)}_1,\dots,\dot{y}_m,\dots,y^{(r_m-1)}_m)^{\top}\in \R^{n_3},\\
N&=\diag(N_1,\dots,N_m)\in \R^{n_3\times n_3} \text{ with } N_i=\left[\SmallNilBlock{0}{1}{2ex}\right]\in \R^{(r_i-1)\times(r_i-1)},\\
E_{32}&=\diag(e_{1}^{[r_1-1]},\dots,e_{1}^{[r_1-1]})\in \R^{n_3\times m},
\end{align*}
where $e_{1}^{[k]}\in \R^{k}$ is the first canonical unit vector.
We note that the above nilpotent matrix $N$ has index $\nu=\max\limits_{1\le i\le m}(r_i-1)$.
\end{enumerate}
\end{remark}

\section{Nonlinear systems with truncated vector relative degree}\label{sect:DAE.nonl.clas}

In this section, we consider a~class of nonlinear DAE systems which comprises the class of linear DAE systems which have a truncated vector relative degree and the same number of inputs and outputs. More precisely, we consider nonlinear functional differential-algebraic systems of the form\index{adaptive!tracking control}\index{tracking control}\index{tracking control!adaptive}\index{system!functional differential-algebraic}\index{system!nonlinear}
\begin{equation}\label{syst:DAE.nonl}
\begin{aligned}
\begin{pmatrix}
y_1^{(r_1)}(t)\\
y_2^{(r_2)}(t)\\
\vdots\\
y_q^{(r_q)}(t)\\
\end{pmatrix}&=f_1\Big(d_1(t), T_1\big(y_1,\ldots,y_1^{(r_1-1)},\ldots,y_q^{(r_q-1)},y_{q+1},\ldots,y_m\big)(t)\Big)\\[-0.7cm]
&\quad +\Gamma_{I}\Big(d_2(t),T_1\big(y_1,\ldots,y_1^{(r_1-1)},\ldots,y_q^{(r_q-1)},y_{q+1},\ldots,y_m\big)(t)\Big)\,u_{I}(t),\\[4mm]
0&=f_2\Big(y_1(t),\ldots,y_1^{(r_1-1)}(t),\ldots,y_q^{(r_q-1)}(t),y_{q+1}(t),\ldots,y_m(t)\Big)\\
&\quad +f_3\Big(d_3(t),(T_2y)(t)\Big)\!+\!\Gamma_{II}\Big(d_4(t),(T_2y)(t)\Big)\,u_{I}(t) \!+\!f_4\Big(d_5(t),(T_2y)(t)\Big)\,u_{II}(t),\\
y|_{[-h,0]}&=y^0
\end{aligned}
\end{equation}
with initial data
\begin{equation}\label{syst:DAE.nonl-IC}
\begin{aligned}
y^0=(y_1^0,y_2^0,\dots,y_m^0)^\top,\quad y_i^0&\in\cC^{r_i-1}([-h,0]\rightarrow \R),\quad i=1,\dots, q,\\
y_i^0&\in\cC([-h,0]\rightarrow \R),\quad i=q+1,\dots, m,
\end{aligned}
\end{equation}
where $f_1,\ldots,f_4$, $\Gamma_I$, $\Gamma_{II}$, $d_1,\ldots, d_5$ are functions and $T_1,T_2$ are operators with properties being specified in the sequel. The output is
$y=(y_1,\ldots,y_m)^\top$ and the input of the system is $u=(u_1,\ldots,u_m)^\top$, for which we set\index{input}\index{output}
\[u_I=(u_1,\dots,u_q)^{\top},\;\; u_{II}=(u_{q+1},\dots,u_m)^{\top},\] i.e., $u=(u_I^\top,u_{II}^\top)^\top$. The functions $d_1,\ldots, d_5:\R_{\ge 0}\to\R^s$ play the roles of disturbances.
We denote $\overline{r}=r_1+\ldots+r_p$
and call -- in virtue of Section~\ref{Sec:norm.form} --  the tuple $(r_1,\dots,r_p,0,\dots,0)\in \N_0^{1\times m}$ with $r_i>0$ for $i=1,\dots,q$ the {\em truncated vector relative degree} of~\eqref{syst:DAE.nonl}.\index{truncated vector relative degree}
We will later show that linear DAE systems which have a truncated vector relative degree belong to this class.
Similar to~\cite{BergIlch14}, we introduce the following classes of operators.

\begin{definition}\label{def:DAE.nonl.oper.T.clas}
For $m,k\in\N$ and $h\ge 0$ the set $\T_{m,k,h}$ denotes the class of operators $T:\cC([-h,\infty)\to\R^m)\to L^\infty_{\loc}(\R_{\ge 0}\to\R^k)$ with the following properties:
\begin{itemize}
\item[(i)]\index{operator}\index{operator!causal} $T$ is causal, i.e, for all $t\ge 0$ and all $\zeta, \xi \in \cC([-h,\infty)\to\R^m)$,
\[\zeta|_{[-h,t)}=\xi|_{[-h,t)}\ \ \Longrightarrow\ \ T(\zeta)|_{[0,t)} = T(\xi)|_{[0,t)}.\]
\item[(ii)]\index{operator!locally Lipschitz continuous}\index{Lipschitz continuous} $T$ is locally Lipschitz continuous in the following sense: for all $t\ge 0$ and all $\xi\in\cC([-h,t]\to\R^{m})$ there exist $\tau, \delta, c>0$ such that, for all $\zeta_1,\zeta_2\in \mathcal{C}([-h,\infty)\to\R^{m})$ with $\zeta_i|_{[-h,t]}=\xi$ and $\|\zeta_i(s)-\xi(t)\|<\delta$ for all $s\in[t,t+\tau]$ and $i=1,2$, we have
    \[
        \left\|\left(T(\zeta_1)-T(\zeta_2)\right)|_{[t,t+\tau]} \right\|_{\infty}        \le c\left\|(\zeta_1-\zeta_2)|_{[t,t+\tau]}\right\|_{\infty}.
    \]
\item[(iii)] $T$ maps bounded trajectories to bounded trajectories, i.e, for all $c_1>0$, there exists $c_2>0$ such that for all $\zeta\in \mathcal{C}([-h,\infty)\to\R^{m})$
\[ \|\zeta|_{[-h,\infty)}\|_\infty\le c_1\ \ \Longrightarrow\ \ \|T(\zeta)|_{[0,\infty)}\|_\infty \le c_2.\]
\end{itemize}
Furthermore, the set $\T_{m,k,h}^{\rm DAE}$ denotes the subclass of operators \[T:\;C([-h,\infty)\to\R^m)\to C^1(\R_{\ge 0}\to\R^k)\] such that $T\in\T_{m,k,h}$ and, additionally,
\begin{itemize}
\item[(iv)] there exist $z\in \cC(\R^m\times \R^k\to\R^k)$ and $\tilde{T}\in\T_{m,k,h}$ such that
$$
\forall\, \zeta\in \cC([-h,\infty)\to\R^m)\ \forall\, t\ge 0:\  \ddts (T\zeta)(t)=z\big(\zeta(t),(\tilde{T}\zeta)(t)\big).
$$
\end{itemize}
\end{definition}

\begin{assumption}\label{def:DAE.nonl.syst.class}
We assume that the functional differential-algebraic system~\eqref{syst:DAE.nonl} has the following properties:
\begin{itemize}
\item[(i)] the gain\index{gain matrix} $\Gamma_{I}\in \cC(\R^s\times \R^k\to\R^{q\times q})$ satisfies $\Gamma_{I}(d,\eta) + \Gamma_{I}(d,\eta)^\top > 0$ for all $(d,\eta)\in\R^s\times\R^k$, and $\Gamma_{II}\in \cC^1(\R^s\times \R^k\to\R^{(m-q)\times q})$.
\item[(ii)] the disturbances\index{disturbance} satisfy $d_1, d_2\in \cL^{\infty}(\R_{\ge 0}\to\R^s)$ and $d_3, d_4, d_5\in W^{1,\infty}(\R_{\ge 0}\to\R^s)$.
\item[(iii)] $f_1\in \cC(\R^s\times \R^k\to\R^q)$, $f_2\in \cC^1(\R^{\overline{r}+m-q}\to\R^{m-q})$, $f_3\in \cC^1(\R^s\times \R^k\to\R^{m-q})$, and $f'_2 \left[\begin{smallmatrix}
0\\
I_{m-q}
\end{smallmatrix}\right]$ is bounded.
\item[(iv)] $f_4\in \cC^1(\R^s\times\R^k\to\R)$ and there exists $\alpha>0$ such that $f_4(d,v)\ge \alpha$ for all $(d,v)\in \R^s\times\R^k$.
\item[(v)] $T_1\in \T_{\overline{r}+m-q,k,h}$ and $T_2\in\T_{m,k,h}^{\rm DAE}$.	
\end{itemize}
\end{assumption}

In the remainder of this section we show that any right-invertible system $[E,A,B,C]\in\Sigma_{l,n,m,m}$ with truncated vector relative degree $(r_1,\ldots,r_q,0,\ldots,0)$, where $r_1,\ldots,r_q\in\N$, belongs to the class of systems~\eqref{syst:DAE.nonl} which satisfy Assumption~\ref{def:DAE.nonl.syst.class} as long as $[E,A,B,C]$ has asymptotically stable zero dynamics and the matrix $\Gamma_{11}$ in~\eqref{DAE.line.norm.form} satisfies $\Gamma_{11}+\Gamma_{11}^\top>0$. We have seen in Remark~\ref{Rem:DAE.syst.inve} that asymptotic stability of the zero dynamics is equivalent to the matrix $Q$ in~\eqref{DAE.line.norm.form} having only eigenvalues with negative real part. 

Consider the three first equations in~\eqref{DAE.line.norm.form} and the operator\label{pp:operatorT2}
$$
\begin{aligned}
T_2:&&\cC([0,\infty)\to\R^m)\to &\;\cC^1(\R_{\ge 0}\to\R^{n_3})\\
&&y\mapsto&\, \left(t\mapsto(T_2y)(t):=\eta(t)=e^{Qt}\eta^0+\int\limits_{0}^{t}e^{Q(t-\tau)} A_{12}y(\tau)d\tau\right),
\end{aligned}$$
which is parameterized by the initial value $\eta^0\in\R^{n_1}$. This operator is clearly causal, locally Lipschitz, and, since all eigenvalues of $Q$ have negative real part, $T$ satisfies property~(iii) in Definition~\ref{def:DAE.nonl.oper.T.clas}. The derivative is given by
$$
\ddts (T_2y)(t)=Qe^{Qt}\eta^0+A_{12}y(t)+Q\int\limits_{0}^{t}e^{Q(t-\tau)} A_{12}y(\tau)d\tau =: (\tilde Ty)(t),\quad t\ge 0,
$$
and it is straightforward to check that $\tilde T\in\T_{m,n_3,0}$. Therefore, we obtain that $T_2\in\T_{m,n_3,0}^{\rm DAE}$. Further consider the operator $T_1:\cC([0,\infty)\to\R^{\overline{r}+m-q})\to L^\infty_{\loc}(\R_{\ge 0}\to\R^{q})$ defined by
\begin{align*}
  &T_1(\zeta_{1,1},\ldots,\zeta_{1,r_1},\ldots,\zeta_{q,r_q},\zeta_{q+1,1},\ldots,\zeta_{m,1})\\
  &= R_{1,1}\begin{pmatrix}\zeta_{1,1}\\\vdots\\\zeta_{1,r_1}\end{pmatrix}+\ldots+R_{q,1}
\begin{pmatrix}\zeta_{q,1}\\\vdots\\\zeta_{q,r_q}\end{pmatrix}+S_1\begin{pmatrix}\zeta_{q+1,1}\\\vdots\\\zeta_{m,1}\end{pmatrix} + P_1 T_2(\zeta_{1,1}, \zeta_{2,1},\ldots, \zeta_{m,1}),
\end{align*}
then, likewise, we obtain that $T_1\in \T_{\overline{r}+m-q,q,0}$. The remaining functions are given by
\begin{align*}
f_1(d,\eta)=\eta,\ \
\Gamma_{I}\big(d,\eta\big)=\Gamma_{11},\ \
f_3\big(d,\eta\big)=P_2 \eta,\ \
\Gamma_{II}\big(d,\eta\big)=\Gamma_{21},\ \
f_4\big(d,\eta\big)=1
\end{align*}
and
\begin{align*}
&f_2(\zeta_{1,1},\ldots,\zeta_{1,r_1},\ldots,\zeta_{q,r_q},\zeta_{q+1,1},\ldots,\zeta_{m,1})\\
  &= R_{1,2}\begin{pmatrix}\zeta_{1,1}\\\vdots\\\zeta_{1,r_1}\end{pmatrix}+\ldots+R_{q,2}
\begin{pmatrix}\zeta_{q,1}\\\vdots\\\zeta_{q,r_q}\end{pmatrix}+S_2\begin{pmatrix}\zeta_{q+1,1}\\\vdots\\\zeta_{m,1}\end{pmatrix}.
\end{align*}
The function $f_2$ satisfies condition (iii) in Assumption~\ref{def:DAE.nonl.syst.class} since
$$
f_2' \begin{bmatrix}
0\\
I_{m-q}
\end{bmatrix}=S_2\in \R^{(m-q)\times (m-q)}.
$$
Note that system \eqref{DAE.line.norm.form} does not entirely belong to the class~\eqref{syst:DAE.nonl} since the fourth equation in \eqref{DAE.line.norm.form} is not included. However, the control objective formulated in the following section can also be achieved for~\eqref{DAE.line.norm.form}, see also Remark~\ref{Rem:funnel-control}~\ref{Rem:funnel-control-e}.

\section{Funnel control}

\subsection{Control objective}\label{subsect:control.objective}
Let a~reference signal\index{reference signal} $y_{\rm ref}=(y_{{\rm ref},1},\ldots,y_{{\rm ref},m})^\top$ with $y_{{\rm ref},i}\in\cW^{r_i,\infty}(\R_{\ge 0}\to\R)$ for $i=1,\ldots,q$ and $y_{{\rm ref},i}\in W^{1,\infty}(\R_{\ge 0}\to\R)$ for $i=q+1,\ldots, m$ be given, and let
$e=y-y_{\rm ref}$ be the {\it tracking error}.
The objective is to design an output error feedback\index{feedback!output error}\index{output error feedback} of the form
\[u(t) = F\bigg(t,e_1(t),\ldots,e_1^{(r_1-1)}(t),\ldots,e_q^{(r_q-1)}(t), e_{q+1}(t),\ldots,e_m(t)\bigg),\]
such that in the closed-loop system the tracking error evolves within a prescribed performance funnel\index{performance funnel}
\begin{equation}
\cF^m_{\varphi} := \setdef{(t,e)\in\R_{\ge 0} \times\R^m}{\varphi(t) \|e\| < 1},\label{eq:perf_funnel}
\end{equation}
which is determined by a function $\varphi$ belonging to
\begin{equation}
\Phi_k:=
\setdef{
\varphi\in  \cC^k(\R_{\ge 0}\to\R)}{
\begin{array}{l}
\varphi, \dot \varphi,\ldots,\varphi^{(k)} \text{ are bounded,}\\
\varphi (\tau)>0 \text{ for all } \tau>0,\text{ and }  \liminf\limits_{\tau\rightarrow \infty} \varphi(\tau) > 0\!\!\!\!
\end{array}}.
\label{eq:Phir}\end{equation}
A~further objective is that all signals $u,e_1,\ldots,e_1^{(r_1-1)},\ldots,e_q^{(r_q-1)}, e_{q+1},\ldots,e_m:\R_{\ge0}\to\R^m$ should remain boun\-ded.

The funnel boundary\index{funnel!boundary} is given by the reciprocal of $\varphi$, see Fig.\ \ref{Fig:funnel}. It is explicitly allowed that $\varphi(0)=0$, meaning that no restriction on the initial value is imposed since $\varphi(0) \|e(0)\| < 1$; the funnel boundary $1/\varphi$ has a pole at $t=0$ in this case.
\begin{figure}[h!tp]
\begin{center}
\begin{tikzpicture}[scale=0.45]
\tikzset{>=latex}
  \filldraw[color=gray!25] plot[smooth] coordinates {(0.15,4.7)(0.7,2.9)(4,0.4)(6,1.5)(9.5,0.4)(10,0.333)(10.01,0.331)(10.041,0.3) (10.041,-0.3)(10.01,-0.331)(10,-0.333)(9.5,-0.4)(6,-1.5)(4,-0.4)(0.7,-2.9)(0.15,-4.7)};
  \draw[thick] plot[smooth] coordinates {(0.15,4.7)(0.7,2.9)(4,0.4)(6,1.5)(9.5,0.4)(10,0.333)(10.01,0.331)(10.041,0.3)};
  \draw[thick] plot[smooth] coordinates {(10.041,-0.3)(10.01,-0.331)(10,-0.333)(9.5,-0.4)(6,-1.5)(4,-0.4)(0.7,-2.9)(0.15,-4.7)};
  \draw[thick,fill=lightgray] (0,0) ellipse (0.4 and 5);
  \draw[thick] (0,0) ellipse (0.1 and 0.333);
  \draw[thick,fill=gray!25] (10.041,0) ellipse (0.1 and 0.333);
  \draw[thick] plot[smooth] coordinates {(0,2)(2,1.1)(4,-0.1)(6,-0.7)(9,0.25)(10,0.15)};
  \draw[thick,->] (-2,0)--(12,0) node[right,above]{\normalsize$t$};
  \draw[thick,dashed](0,0.333)--(10,0.333);
  \draw[thick,dashed](0,-0.333)--(10,-0.333);
  \node [black] at (0,2) {\textbullet};
  \draw[->,thick](4,-3)node[right]{\normalsize$\lambda$}--(2.5,-0.4);
  \draw[->,thick](3,3)node[right]{\normalsize$(0,e(0))$}--(0.07,2.07);
  \draw[->,thick](9,3)node[right]{\normalsize$\varphi(t)^{-1}$}--(7,1.4);
\end{tikzpicture}
\end{center}
\caption{Error evolution in a funnel $\cF^1_{\varphi}$ with boundary $\varphi(t)^{-1}$ for $t > 0$.}
\label{Fig:funnel}
\end{figure}
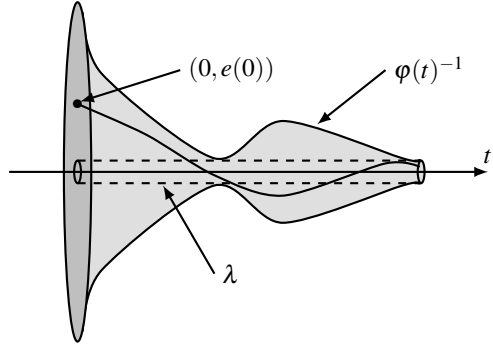
Since every $\varphi\in\Phi_k$ is bounded, the boundary of the associated performance funnel $\cF^m_{\varphi}$ is
bounded away from zero, which means that there exists $\lambda>0$ with $1/\varphi(t)\geq\lambda$ for all $t>0$. Further note that the funnel boundary is not necessarily monotonically decreasing, but it might be beneficial to choose a wider funnel over some later time interval, for instance in the presence of periodic disturbance or when the reference signal varias strongly. Various different funnel boundaries are possible, see e.g.~\cite[Sec.~3.2]{Ilch13}.

\subsection{Controller design}\label{sect:DAE.nonl.funn}

\index{funnel!controller}
The funnel controller for systems of the form \eqref{syst:DAE.nonl} satisfying Assumption~\ref{def:DAE.nonl.syst.class} is of the following form:
\begin{equation}\label{DAE.nonl.funn.controller}
\boxed{
\begin{array}{l}
\mbox {For } i=1,\dots,q:\\
\begin{array}{rcl}
e_{i0}(t)&=&e_i(t)= y_i(t)-y_{\rmref,i}(t),\\
e_{i1}(t)&=&\dot{e}_{i0}(t)+k_{i0}(t)e_{i0}(t),\\
e_{i2}(t)&=&\dot{e}_{i1}(t)+k_{i1}(t)e_{i1}(t),\\
&\vdots& \\
e_{i,r_i-1}(t)&=&\dot{e}_{i,r_i-2}(t)+k_{i,r_i-2}(t)e_{i,r_i-2}(t),\\[2mm]
k_{ij}(t)&=&\tfrac{1}{1-\varphi_{ij}^2(t)|e_{ij}(t)|^2},\quad j=0,\dots,r_i-2.\\[4mm]
\end{array}\\
\mbox {For } i=q+1,\dots,m:\ \  e_i(t)= y_i(t)-y_{\rmref,i}(t),\\[2mm]
\begin{aligned}
e_{I}(t)=&(e_{1,r_1-1}(t),\dots,e_{q,r_q-1}(t))^{\top},& e_{II}(t)=&(e_{q+1}(t),\dots,e_m(t))^{\top},\\
k_{I}(t)=&\tfrac{1}{1-\varphi_{I}(t)^2\|e_{I}(t)\|^2},& k_{II}(t)=&\tfrac{\hat{k}}{1-\varphi_{II}(t)^2\|e_{II}(t)\|^2},\\
\end{aligned}\\[5mm]
u(t)=\begin{pmatrix}
u_I(t)\\
u_{II}(t)
\end{pmatrix}=
\begin{pmatrix}
-k_{I}(t)e_{I}(t)\\
-k_{II}(t)e_{II}(t)
\end{pmatrix},
\end{array}
}
\end{equation}
where we impose the following conditions on the reference signal and funnel functions:
\begin{equation}\label{DAE.line.funn.controller.ass}
\boxed{
\begin{aligned}
y_{\rmref}&=(y_{\rmref,1},\dots,y_{\rmref,m})^\top, \quad y_{\rmref,i}\in \cW^{r_i,\infty}(\R_{\ge 0}\rightarrow \R),\ i=1,\ldots,q\\
&\qquad\qquad\qquad\qquad y_{\rmref,i}\in \cW^{1,\infty}(\R_{\ge 0}\rightarrow \R),\ i=q+1,\ldots,m\\
\varphi_I,\varphi_{II}&\in \Phi_{1},\ \varphi_{ij}\in \Phi_{r_i-j},\ i=1,\dots,q,\ j=0,\ldots,r_i-2.
\end{aligned}
}
\end{equation}
We further assume that $\hat{k}$ satisfies
\begin{equation}\label{DAE.nonl.gain.condition}
\hat{k}>\alpha^{-1}\sup\limits_{Y\in \R^{\overline{r}+m-q}}\left\|f'_2(Y)\begin{bmatrix}
0\\
I_{m-q}
\end{bmatrix}\right\|.
\end{equation}

\newpage

\begin{remark}\label{Rem:funnel-control}\
\begin{enumerate}[a)]
\item By a solution of the closed-loop system\index{closed-loop}\index{system!closed-loop} \eqref{syst:DAE.nonl}, \eqref{DAE.nonl.funn.controller} on $[-h,\omega)$, $\omega\in(0,\infty]$, with initial data $y^0$ as in~\eqref{syst:DAE.nonl-IC} we mean a function $y=(y_1,\dots,y_m)^\top$ such that $y|_{[-h,0]}=y^0$, $y_i\in \cC^{r_i-1}([-h,\omega)\rightarrow \R)$ and $y_i^{(r_i-1)}|_{[0,\omega)}$ is weakly differentiable for $i=1,\dots,q$,  $y_i\in \cC([-h,\omega]\rightarrow \R)$ and $y_i|_{[0,\omega)}$ is weakly differentiable for $i=q+1,\dots, m$, and~$y$ satisfies the differential-algebraic equation in~\eqref{syst:DAE.nonl} with $u$ defined in~\eqref{DAE.nonl.funn.controller} in the weak sense. The solution $y$ is called {\it maximal}, if it has no right extension that is also a solution, and {\em global}, if $\omega=\infty$.\index{solution!maximal}\index{solution!global}

\item Assumption~\ref{def:DAE.nonl.syst.class}~(iii) together with condition~\eqref{DAE.nonl.gain.condition} are essential for the solvability of the closed-loop system \eqref{syst:DAE.nonl}, \eqref{DAE.nonl.funn.controller}, since they guarantee the invertibility of $\alpha\hat{k}I_{m-q}-f'_3(Y)\left[\begin{smallmatrix}
0\\
I_{m-q}
\end{smallmatrix}\right]$. This property is crucial for the explicit solution of the algebraic constraint\index{algebraic constraint} in the closed-loop system~\eqref{syst:DAE.nonl},~\eqref{DAE.nonl.funn.controller}.

\item If the system \eqref{syst:DAE.nonl} has strict relative, i.e., $q=m$ and $r_1=\ldots=r_m=:r>0$, then it satisfies the assumptions of~\cite[Thm.~3.1]{BergLe18a}. In this case, the funnel controller~\eqref{DAE.nonl.funn.controller} simplifies to
\[\boxed{
\begin{array}{l}
\mbox {For } i=1,\dots,m,\\
\begin{array}{rcl}
e_{i0}(t)&=&e_i(t)=y_i(t)-y_{\rmref,i}(t),\\
e_{i1}(t)&=&\dot{e}_{i0}(t)+k_{i0}(t)e_{i0}(t),\\
e_{i2}(t)&=&\dot{e}_{i1}(t)+k_{i1}(t)e_{i1}(t),\\
&\vdots&  \\
e_{i,r-1}(t)&=&\dot{e}_{i,r-2}(t)+k_{i,r-2}(t)e_{i,r-2}(t),\\[1mm]
k_{ij}(t)&=&\tfrac{1}{1-\varphi_{ij}^2(t)|e_{ij}(t)|^2},\ \ j=0,\dots,r-2, \\
e_{r-1}(t)&=& (e_{1,r-1}(t),\ldots,e_{m,r-1}(t))^\top\\
k_{r-1}(t)&=&\tfrac{1}{1-\varphi_{r-1}(t)^2\|e_{r-1}(t)\|^2},\\
u(t)&=&-k_{r-1}(t)e_{r-1}(t).
\end{array}
\end{array}
}
\]
This controller slightly differs from the one presented in~\cite{BergLe18a} for systems with strict relative degree (even when we choose $\varphi_{ij} = \varphi_{1j}$ for all $i=1,\ldots,m$), which reads
\begin{equation}\label{ODE.stri.funn.controller}
\boxed{\begin{array}{l}
\begin{array}{rcl}
e_0(t)&=&e(t) = y(t)-y_{\rmref}(t),\\
e_1(t)&=&\dot{e}_0(t)+k_0(t)e_0(t),\\
e_2(t)&=&\dot{e}_1(t)+k_1(t)e_1(t),\\
&\vdots&  \\
e_{r-1}(t)&=&\dot{e}_{r-2}(t)k_{r-2}(t)e_{r-2}(t),\\[2mm]
k_i(t)&=&\tfrac{1}{1-\varphi_i(t)^2\|e_i(t)\|^2},\ \ i=0,\dots,r-1, \\
u(t)&=&
-k_{r-1}(t)e_{r-1}(t).
\end{array}
\end{array}
}
\end{equation}

\item If the system \eqref{syst:DAE.nonl} satisfies $q=0$, then the funnel controller~\eqref{DAE.nonl.funn.controller}  simplifies to
$$
\begin{aligned}
e(t)&=y(t)-y_{\rmref}(t),\qquad k(t)=\tfrac{\hat{k}}{1-\varphi(t)^2\|e(t)\|^2},\\
u(t)&=-k(t)e(t),
\end{aligned}
$$
and feasibility follows from the results in \cite{BergIlch14} where funnel control for this type has been considered.

\item\label{Rem:funnel-control-e}
Let us stress again that a linear system of the form~\eqref{DAE.line.norm.form} does not completely belong to the class~\eqref{syst:DAE.nonl} as the fourth equation in~\eqref{DAE.line.norm.form} is not included. However, we like to emphasize that in
$$
x_3(t)=\sum\limits_{i=0}^{\nu-1}N^{i}E_{32}y^{(i+1)}(t),
$$
the output $y$ is required smooth enough for $x_3$ to be well defined. Nevertheless, the funnel controller~\eqref{DAE.nonl.funn.controller} can also be applied to systems of the form~\eqref{DAE.line.norm.form}. To see this, assume that there exists a solution to~\eqref{DAE.nonl.funn.controller} applied to~\eqref{DAE.line.norm.form} except for the fourth equation. If the funnel functions $\varphi_I, \varphi_{II}$ and $\varphi_{ij}$, $i=1,\ldots,q$, $j=0,\ldots,r_i-2$ are additionally in $C^{\nu+1}(\R_{\ge 0}\to\R)$ and $y_{\rm ref}$ is additionally in $W^{\nu+2,\infty}(\R_{\ge 0}\to\R^m)$, then the solution $y|_{[0,\infty)}$ will be at least in $C^{\nu+1}(\R_{\ge 0}\to\R^m)$, so that $x_3$ is well defined and continuously differentiable. The proof of this statement is similar to Step~2 of the proof of~\cite[Thm.~5.3]{Berg16b}. Furthermore, using $y_{\rm ref}\in W^{\nu+2,\infty}(\R_{\ge 0}\to\R^m)$ also yields boundedness of~$x_3$, cf.\ Step~4 of the proof of~\cite[Thm.~5.3]{Berg16b}.
%
\end{enumerate}
\end{remark}

\begin{remark}\label{Rem:cons-init-val}
Consider a system~\eqref{syst:DAE.nonl} which satisfies Assumption~\ref{def:DAE.nonl.syst.class} and let the reference signal and funnel functions be as in~\eqref{DAE.line.funn.controller.ass}. Since the second equation in~\eqref{syst:DAE.nonl} is an algebraic equation we need to guarantee that it is initially satisfied for a solution to exist. Since $T_2\in\T_{m,k,h}^{\rm DAE}$ is causal it ``localizes'', in a natural way, to an operator $\hat T_2: C([-h,\omega]\to\R^n)\to C^1([0,\omega]\to\R^k)$, cf.~\cite[Rem.~2.2]{IlchRyan09}. With some abuse of notation, we will henceforth not distinguish between~$T_2$ and its ``localization'' $\hat T_2$. Note that for $\omega=0$ we have that $\hat T_2: C([-h,0]\to\R^n)\to \R^k$. Hence, an initial value $y^0$ as in~\eqref{syst:DAE.nonl-IC} is called {\it consistent}\index{initial value}\index{initial value!consistent} for the closed loop system~\eqref{syst:DAE.nonl},~\eqref{DAE.nonl.funn.controller}, if
\begin{equation}\label{eq:consistent-IV}
\begin{aligned}
&f_2\left(y_1^0(0),\dots,\big(\ddts\big)^{r_1-1}(y_1^0)(0),\dots, \big(\ddts\big)^{r_q-1}(y_q^0)(0),y_{q+1}^0(0),\ldots,y_m^0(0)\right)\\
&+f_3\big(d_3(0),T_2(y^0)\big)+\Gamma_{II}\big(d_4(0),T_2(y^0)\big)u_I(0)+f_4\big(d_5(0),T_2(y^0)\big)u_{II}(0)=0,
\end{aligned}
\end{equation}
where $u_I(0), u_{II}(0)$ are defined by \eqref{DAE.nonl.funn.controller}.
\end{remark}

\subsection{Feasibility of funnel control}\label{sect:feasible.funnel}

We show feasibility of the funnel controller \eqref{DAE.nonl.funn.controller} for systems of the form~\eqref{syst:DAE.nonl} satisfying Assumption~\ref{def:DAE.nonl.syst.class}. The following theorem unifies and extends the funnel control results from~~\cite{Berg16b,BergIlch12a,BergIlch12b,BergIlch14,BergLe18a}, which are all special cases of it.

\begin{theorem}\label{Thm:DAE.nonl.funn.controller}
Consider a~system \eqref{syst:DAE.nonl} satisfying Assumption~\ref{def:DAE.nonl.syst.class}. Let $y_{\rmref}$ and $\varphi_I, \varphi_{II}, \varphi_{ij}$, $i=1,\dots,q$, $j=0,\dots,r_i-2$ be as in~\eqref{DAE.line.funn.controller.ass} and $\hat k>0$ such that~\eqref{DAE.nonl.gain.condition} holds. Then for any consistent initial value $y^0$ as in~\eqref{syst:DAE.nonl-IC} (i.e., $y^0$ satisfies~\eqref{eq:consistent-IV}) such that $e_{I},e_{II},e_{ij}$, $i=1,\dots,q$, $j=0,\dots,r_i-2$ defined in~\eqref{DAE.nonl.funn.controller} satisfy
\begin{equation}\label{DAE.nonl.init.condition}
\begin{aligned}
\varphi_{I}(0)\|e_I(0)\|\,&<1,\quad \varphi_{II}(0)\|e_{II}(0)\|\,<1,\\
\varphi_{ij}(0)|e_{ij}(0)|\,&<1,\quad i=1,\dots,q,\ j=0,\dots,r_i-2,
\end{aligned}\end{equation}
the application of the funnel controller \eqref{DAE.nonl.funn.controller} to \eqref{syst:DAE.nonl} yields a closed-loop initial value problem that has a solution and every solution can be extended to a global solution. Furthermore, for every global solution $y(\cdot)$,
\begin{enumerate}
\item the input $u:\R_{\ge0}\to\R^m$ and the gain functions $k_{I},k_{II}, k_{ij}:\R_{\ge0}\to\R$, $i=1,\dots,q$, $j=0,\dots,r_i-2$ are bounded;
\item the functions $e_{I}:\R_{\ge 0}\to\R^q$, $e_{II}:\R_{\ge 0}\to\R^{m-q}$ and $e_{ij}:\R_{\ge 0}\to\R$, $i=1,\dots,q$, $j=0,\dots,r_i-2$ evolve in their respective performance funnels, i.e.,  for all $i=1,\dots,q$, $j=0,\dots,r_i-2$ and $t\ge 0$ we have
$$
(t,e_{I}(t))\in\cF^q_{\varphi_{I}},\ (t,e_{II}(t))\in\cF^{m-q}_{\varphi_{II}},\ (t,e_{ij}(t))\in\cF^1_{\varphi_{ij}}.
$$
Furthermore, the signals $e_{I}(\cdot),e_{II}(\cdot),e_{ij}(\cdot)$ are uniformly bounded away from the funnel boundaries in the following sense:
\begin{equation}\label{DAE.nonl.bounded.away.error}
\begin{aligned}
&\exists\, \varepsilon_{I}>0\;  \forall\, t>0\,:&\|e_{I}(t)\| & \le \varphi_{I}(t)^{-1} - \varepsilon_{I},\\
&\exists\, \varepsilon_{II}>0\;  \forall\, t>0\,:&\|e_{II}(t)\| & \le \varphi_{II}(t)^{-1} - \varepsilon_{II},\\
&\forall\, i=1,\dots,q,\ j=0,\dots,r_i-2\; \exists\, \varepsilon_{ij}>0\  \forall\, t>0:\!\!\!\!\!\!\!\!&|e_{ij}(t)| & \le \varphi_{ij}(t)^{-1} - \varepsilon_{ij}.
\end{aligned}
\end{equation}
In particular, each error component $e_i(t)= y_i-y_{\rmref,i}(t)$ evolves in the funnel $\cF^1_{\varphi_{i0}}$, for $i=1,\dots,q$, or $\cF^1_{\varphi_{II}}$, for $i=q+1,\dots,m$, resp., and stays uniformly away from its boundary.
\end{enumerate}
\end{theorem}

The proof of this theorem is similar to the one of \cite[Thm.~3.1]{BergLe18a}, where the feasibility of the funnel controller~\eqref{ODE.stri.funn.controller} for ODE systems with strict relative degree has been treated. However, one of the additional difficulties in proving this theorem is that the closed-loop system~\eqref{syst:DAE.nonl}, \eqref{DAE.nonl.funn.controller} is now a DAE because of the second equation in~\eqref{syst:DAE.nonl}.

\begin{proof}
We proceed in several steps.\\
{\em Step 1}: We show that a maximal solution $y:[-h,\omega)\to \R^m$, $\omega\in (0,\infty]$, of the closed-loop system \eqref{syst:DAE.nonl}, \eqref{DAE.nonl.funn.controller} exists. To this end, we seek to reformulate~\eqref{syst:DAE.nonl},~\eqref{DAE.nonl.funn.controller} as an initial value problem of the form
\begin{equation}\label{syst:DAE.nonl.functional.closed.loop}
\begin{aligned}
\dot{X}_I(t)&=F_{I}\left(t,\begin{pmatrix}
X_I(t)\\
X_{II}(t)
\end{pmatrix},T_1\begin{pmatrix}
X_I\\
X_{II}
\end{pmatrix}(t) \right),\\
0&=F_{II}\left(t,\begin{pmatrix}
X_I(t)\\
X_{II}(t)
\end{pmatrix},\hat T_2\begin{pmatrix}
X_I\\
X_{II}
\end{pmatrix}(t) \right)
\end{aligned}
\end{equation}
with
\begin{equation}\label{syst:DAE.nonl.functional.closed.loop-IC}
\begin{aligned}
X_I|_{[-h,0]}&=\left(y_1^0,\dots,\big(\ddts\big)^{r_1-1}y_1^0,\dots, \big(\ddts\big)^{r_q-1}y_q^0\right)^\top,\\
X_{II}|_{[-h,0]}&= \big(y_{q+1}^0,\ldots,y_m^0\big)^\top.
\end{aligned}
\end{equation}
{\em Step 1a}: Define, for $i=1,\dots,q$, and $j=0,\dots,r_i-2$, the sets
$$
\cD_{ij}:=\setdef{(t,e_{i0},\dots,e_{ij})\in \R_{\ge 0}\times\R\times\dots\times\R}{(t,e_{i\ell})\in \cF^1_{\varphi_{i\ell}},\ell=0,\ldots,j},
$$
where $\cF^1_{\varphi_{i\ell}}$ is as in \eqref{eq:perf_funnel}, and the functions $K_{ij}:\cD_{ij}\to\R$ recursively by
\begin{align*}
K_{i0}(t,e_{i0})&:= \tfrac{e_{i0}}{1-\varphi_{i0}^2(t)|e_{i0}|^2},\\
K_{ij}(t,e_{i0},\ldots,e_{ij})&:= \tfrac{e_{ij}}{1-\varphi_{ij}^2(t)|e_{ij}|^2} + \tfrac{\partial K_{i,j-1}}{\partial t}(t,e_{i0},\ldots,e_{i,j-1})\\
  &\quad + \sum_{\ell=0}^{j-1} \tfrac{\partial K_{i,j-1}}{\partial e_{\ell j}}(t,e_{i0},\ldots,e_{i,j-1}) \left( e_{i,\ell+1} - \tfrac{e_{i\ell}}{1-\varphi_{i\ell}^2(t)|e_{i\ell}|^2}\right).
\end{align*}
Now recall that $\overline{r} = r_1 + \ldots + r_q$ and set
\begin{align*}
\cD_{I}&:=\left\{(t,e_{10},\dots,e_{1,r_1-1},\dots,e_{q,r_q-1})\in \R_{\ge 0}\times\R^{\overline{r}}\;\right|\\
&\qquad\left.\phantom{\R^{\overline{r}}}\forall\,i=1,\ldots,q:\ \big(t,e_{i0},\dots,e_{i,r_i-2}\big)\in \mathcal{D}_{i,r_i-2}\ \wedge\ (t,e_{1,r_1-1},\ldots,e_{q,r_q-1})\in \cF^q_{\varphi_{I}}\right\},\\
\cD_{II}&:= \cF^{m-q}_{\varphi_{II}},\\
\cD&:=\setdef{(t,e_I,e_{II})\in\R_{\ge 0}\times\R^{\overline{r}}\times\R^{m-q}}{(t,e_I)\in \cD_I\ \wedge\ (t,e_{II})\in \cD_{II}}.
\end{align*}
Choose some interval $I\subseteq\R_{\ge0}$ with $0\in I$ and let \[(e_{10},\dots,e_{1,r_1-1},\dots,e_{q,r_q-1}):I\to\R^{\overline{r}}\]
be sufficiently smooth such that for all $t\in I$ we have
\begin{align*}
\big(t,e_{10}(t),\dots,e_{1,r_1-1}(t),\dots,e_{q,r_q-1}(t)\big)&\in \cD_{I},\\
(t,e_{q+1}(t),\dots,e_m(t))&\in \cD_{II}
\end{align*}
 and $(e_{i0},\ldots,e_{i,r_i-1})$, $i=1,\dots,q$, satisfies the relations in \eqref{DAE.nonl.funn.controller}. Then $e_i=e_{i0}$ satisfies, on the interval $I$,
\begin{equation}\label{eq:DAE.nonl.eiqrel}
e_i^{(j)}=e_{ij}-\sum\limits_{\ell=0}^{j-1}\left(\ddts\right)^{j-1-\ell}k_{i\ell}e_{i\ell},\quad i=1,\ldots,q,\ j=1,\dots, r_i-1.
\end{equation}
{\em Step 1b:} We show by induction that for all $i=1,\dots,q$, and $j=0,\ldots,r_i-2$ we have
\begin{equation}\label{eq:DAE.nonl.vec.sumji}
    \forall\, t\in I:\ \sum\limits_{\ell=0}^{j}  \left(\ddts\right)^{j-\ell} \Big(k_{i\ell}(t)e_{i\ell}(t)\Big) = K_{ij}\Big(t,e_{i0}(t),\ldots,e_{ij}(t)\Big).
\end{equation}
Fix $t\in I$. Equation \eqref{eq:DAE.nonl.vec.sumji} is obviously true for $j=0$. Assume that $j\in\{1,\ldots,r_i-2\}$ and the statement holds for $j-1$. Then
\begin{align*}
\sum\limits_{\ell=0}^{j}  \left(\ddts\right)^{j-\ell} \Big(k_{i\ell}(t)e_{i\ell}(t)\Big)&= k_{ij}(t)e_{ij}(t) + \ddts \left(\sum\limits_{\ell=0}^{j-1}  \left(\ddts\right)^{j-\ell-1} \Big(k_{i\ell}(t)e_{i\ell}(t)\Big)\right)\\
&= k_{ij}(t)e_{ij}(t) + \ddts K_{i,j-1}\Big(t,e_{i0}(t),\ldots,e_{i,j-1}(t)\Big) \\
&= K_{ij}\Big(t,e_{i0}(t),\ldots,e_{ij}(t)\Big).
\end{align*}
Therefore,~\eqref{eq:DAE.nonl.vec.sumji} is shown and, invoking~\eqref{eq:DAE.nonl.eiqrel}, we have for all $i=1,\ldots,q$ and $t\in I$ that
\begin{equation}\label{eq:DAE.nonl.vec.eiqform}
e_i^{(j)}(t)=e_{ij}(t)-K_{i,j-1}\big(t,e_{i0}(t),\ldots,e_{i,j-1}(t)\big),\quad j=1,\dots, r_i-1.
\end{equation}
{\em Step 1c:} Define, for $i=1,\dots,q$,
$$
\tilde K_{i0} :\R_{\ge 0}\times\R \to \R,\ (t,y_{i0})\mapsto y_{i0}-y_{\rmref,i}(t)
$$
and the set
$$
\tilde \cD_{i0} := \setdef{(t,y_i)\in\R_{\ge 0}\times\R}{\big(t,\tilde K_{i0}(t,y_i)\big)\in\cD_{i0}}.
$$
Furthermore, recursively define the maps
\begin{align*}
\tilde K_{ij}:\tilde \cD_{i,j-1} \times\R &\to \R,\\
(t,y_{i0},\ldots,y_{ij})&\mapsto y_{ij} - y_{\rmref,i}^{(j)}(t)+ K_{i,j-1}\big( t, \tilde K_{i0}(t,y_{i0}), \ldots, \tilde K_{i,j-1}(t,y_{i0},\ldots,y_{i,j-1})\big),
\end{align*}
for $j=1,\ldots,r_i-1$ and the sets
\[
\tilde \cD_{ij}:=\setdef{(t,y_{i0},\dots,y_{ij})\in \tilde\cD_{i,j-1}\times\R}{\big( t, \tilde K_{i0}(t,y_{i0}), \ldots,\tilde K_{ij}(t,y_{i0},\ldots,y_{ij})\big)\in\cD_{ij}}
\]
for $j=1,\ldots,r_i-2$. Then it follows from~\eqref{eq:DAE.nonl.vec.eiqform} and a simple induction that for all $t\in I$, $i=1,\dots,q$, and $j=0,\ldots,r_i-1$ we have
$$
e_{ij}(t)=\tilde{K}_{ij}\big(t,y_i(t),\dots,y_i^{(j)}(t)\big).
$$
Now, define
\begin{align*}
\tilde \cD_{I}&:=\left\{(t,y_{10},\dots,y_{1,r_1-1},\dots,y_{q,r_q-1})\in\R_{\ge 0}\times\R^{\overline{r}}\ \right|\\
&\qquad\quad\forall\,i=1,\ldots,q:\ (t,y_{i0},\dots,y_{i,r_i-1})\in \tilde\cD_{i,r_i-2}\times\R\\
&\qquad\quad\left.\wedge\ \big( t, \tilde K_{1,r_1-1}(t,y_{10},\ldots,y_{1,r_1-1}), \ldots, \tilde K_{q,r_q-1}(t,y_{q0},\ldots,y_{q,r_q-1})\big)\in\cF_{\varphi_I}^q\right\},\\
\tilde \cD_{II}&:=\setdef{(t,y_{q+1},\dots,y_m)\in \R_{\ge 0}\times\R^{m-q}}{(t,y_{q+1}-y_{\rmref,q+1}(t),\dots,y_m-y_{\rmref,m}(t))\in \cD_{II}},
\end{align*}
and the map
\begin{align*}
\tilde K_I:\tilde \cD_{I} \to \R^q,\ & (t,y_{10},\ldots,y_{1,r_1-1},\ldots,y_{q,r_q-1})\\
&\mapsto \left( \tilde K_{1,r_1-1}(t,y_{10},\ldots,y_{1,r_1-1}), \ldots, \tilde K_{q,r_q-1}(t,y_{q0},\ldots,y_{q,r_q-1})\right)^\top,
\end{align*}
then we find that, for all $t\in I$,
\[
    e_I(t):= \big(e_{1,r_1-1}(t),\ldots,e_{q,r_q-1}(t)\big)^\top = \tilde K_I\left(t,y_1(t),\ldots,y_1^{(r_1-1)}(t),\ldots,y_q^{(r_q-1)}(t)\right).
\]
Further  denote, for $t\in I$,
\begin{align*}
X_{I}(t)&= \Big(y_1(t),\dots,y_1^{(r_1-1)}(t)\dots, y_q^{(r_q-1)}(t)\Big)^{\top},&
X_{II}(t)&=(y_{q+1}(t),\dots,y_m(t))^{\top},\\
X_{\rmref,II}(t)&=(y_{\rmref,q+1}(t),\dots,y_{\rmref,m}(t))^{\top},
\end{align*}
then
\begin{align*}
    e_I(t)&=\tilde{K}_{I}(t,X_{I}(t)),\\
e_{II}(t)&:=(y_{q+1}(t)-y_{\rmref,q+1}(t),\dots,y_m(t)-y_{\rmref,m}(t))^{\top} = X_{II}(t)-X_{\rmref,II}(t)
\end{align*}
and the feedback $u$ in \eqref{DAE.nonl.funn.controller} reads
$$
u(t)=\begin{pmatrix}
\frac{-\tilde{K}_{I}(t,X_{I}(t))}{1-\varphi_{I}(t)^2\|\tilde{K}_{I}(t,X_{I}(t))\|^2}\\[3mm]
\frac{-\hat{k} (X_{II}(t)-X_{\rmref,II}(t))}{1-\varphi_{II}(t)^2\|X_{II}(t)-X_{\rmref,II}(t)\|^2}
\end{pmatrix}.
$$
{\em Step 1d}:
Now, we set
\begin{equation}\label{eq:DAE.nonl.S.matrix}
\begin{aligned}
H&=\diag\left((e_1^{[r_1]})^\top,\dots,(e_1^{[r_q]})^\top\right)\in \R^{q\times \overline{r}},\\
S&=\begin{bmatrix}
H&0\\
0&I_{m-q}
\end{bmatrix}\in \R^{m\times (\overline{r}+m-q)},
\end{aligned}
\end{equation}
where $e_1^{[k]}\in\R^{k}$ is the first canonical unit vector.
This construction yields
$$
\forall\, t\in I:\quad S\begin{pmatrix}
X_I(t)\\
X_{II}(t)
\end{pmatrix}=y(t).
$$
We define an operator $\hat T_2:\cC([-h,\infty)\to\R^{\overline{r}+m-q})\to \cC^1(\R_{\ge 0}\to\R^k)$ such that for $\zeta_1\in \cC([-h,\infty)\to\R^{\overline{r}})$, $\zeta_2\in\cC([-h,\infty)\to\R^{m-q})$ we have
$$
\hat T_2\begin{pmatrix}
\zeta_1\\ \zeta_2
\end{pmatrix}(t) := T_2\left(S\begin{pmatrix}
\zeta_1\\ \zeta_2
\end{pmatrix} \right)(t),\quad t\ge 0.
$$
Since $T_2\in\T_{m,k,h}^{\rm DAE}$ we obtain that $\hat T_2\in\T_{\overline{r}+m-q,k,h}^{\rm DAE}$.
Set
$$
\tilde{\cD}:=\setdef{(t,X_I,X_{II})\in\R_{\ge 0}\times\R^{\overline{r}}\times\R^{m-q}}{(t,X_I)\in \tilde{\cD}_{I}\text{ and } (t,X_{II})\in \tilde{\cD}_{II}}.
$$
We rewrite $f_1$, and $\Gamma_I$ from system \eqref{syst:DAE.nonl} in vector form
\[f_1=\begin{pmatrix}
f_1^1\\
\vdots\\
f_1^q
\end{pmatrix},\quad
\Gamma_I=\begin{pmatrix}
\Gamma_I^1\\
\vdots\\
\Gamma_I^q
\end{pmatrix}
\]
with components
$f_1^i\in\cC(\R^s\times\R^k\to \R)$ and
$\Gamma_I^i\in\cC(\R^s\times\R^k\to \R^{1\times q})$ for $i=1,\dots,q$. We now define functions $F_I:\tilde{\cD}\times\R^k \to\R^{\overline{r}}$, $F_{II}:\tilde{\cD}\times\R^k \to\R^{m-q}$ with
\begin{multline*}
F_I:\quad (t,\underset{=X_I}{\underbrace{y_{10},\dots,y_{1,r_1-1},\dots,y_{q,r_q-1}}}, \underset{=X_{II}}{\underbrace{y_{q+1},\dots,y_m}},\eta)\mapsto\\
\Bigg(y_{11},\ldots,y_{1,r_1-1},f_1^1(d_1(t),\eta)- \dfrac{\Gamma_I^1(d_2(t),\eta)\tilde{K}_{I}(t,X_I)}{1-\varphi_{I}(t)^2\|\tilde{K}_{I}(t,X_I)\|^2},\\
\ldots, y_{q,r_q-1}, f_1^q(d_1(t),\eta)-\dfrac{\Gamma_I^q(d_2(t),\eta)\tilde{K}_{I}(t,X_I)}{1-\varphi_{I}(t)^2\|\tilde{K}_{I}(t,X_I)\|^2}\Bigg),
\end{multline*}
\begin{multline*}
F_{II}:\quad (t,\underset{=X_I}{\underbrace{y_{10},\dots,y_{1,r_1-1},\dots,y_{q,r_q-1}}}, \underset{=X_{II}}{\underbrace{y_{q+1},\dots,y_m}},\eta)\mapsto\\
\Bigg(f_2(X_I,X_{II})+f_3(d_3(t),\eta)- \dfrac{\Gamma_{II}(d_4(t),\eta)\tilde{K}_{I}(t,X_I)}{1-\varphi_{I}(t)^2\|\tilde{K}_{I}(t,X_I)\|^2}\\
-f_4(d_5(t),\eta)\dfrac{\hat{k}\, (X_{II}-X_{\rmref,II}(t))}{1-\varphi_{II}(t)^2\|X_{II}-X_{\rmref,II}(t)\|^2}\Bigg).
\end{multline*}
Then the closed-loop system~\eqref{syst:DAE.nonl},~\eqref{DAE.nonl.funn.controller}  is equivalent to~\eqref{syst:DAE.nonl.functional.closed.loop}.\\
{\em Step 1e}: In order to show that~\eqref{syst:DAE.nonl.functional.closed.loop} has a solution we take the derivative of the second equation and rewrite it appropriately. First observe that since $T_2\in\T_{m,k,h}^{\rm DAE}$ there exist $z\in C(\R^m\times\R^k\to\R^k)$ and $\tilde T_2\in \T_{m,k,h}$ such that
\[
    \forall\, \zeta\in \cC([-h,\infty)\to\R^{m})\ \forall\ t\ge 0:\ \ddts (T_2\zeta)(t)=z\big(\zeta(t),(\tilde{T}_2\zeta)(t)).
\]
Now define the operator $\hat T_3:\cC([-h,\infty)\to\R^{\overline{r}+m-q})\to \cC(\R_{\ge 0}\to\R^k)$ by the property that for $\zeta_1\in \cC([-h,\infty)\to\R^{\overline{r}})$, $\zeta_2\in\cC([-h,\infty)\to\R^{m-q})$ we have
$$
\hat T_3\begin{pmatrix}
\zeta_1\\ \zeta_2
\end{pmatrix}(t) := \tilde T_2\left(S\begin{pmatrix}
\zeta_1\\ \zeta_2
\end{pmatrix} \right)(t),\quad t\ge 0,
$$
then $\hat T_3\in \T_{\overline{r}+m-q,k,h}$. A differentiation of the second equation in~\eqref{syst:DAE.nonl.functional.closed.loop} yields
\begin{align*}
    0 &= \frac{\partial F_{II}}{\partial t}\left(t,\begin{pmatrix}
X_I(t)\\
X_{II}(t)
\end{pmatrix},\hat T_2\begin{pmatrix}
X_I\\
X_{II}
\end{pmatrix}(t) \right)
+ \frac{\partial F_{II}}{\partial X_I}\left(t,\begin{pmatrix}
X_I(t)\\
X_{II}(t)
\end{pmatrix},\hat T_2\begin{pmatrix}
X_I\\
X_{II}
\end{pmatrix}(t) \right)\, \dot X_I(t) \\
&\quad + \frac{\partial F_{II}}{\partial X_{II}}\left(t,\begin{pmatrix}
X_I(t)\\
X_{II}(t)
\end{pmatrix},\hat T_2\begin{pmatrix}
X_I\\
X_{II}
\end{pmatrix}(t) \right)\, \dot X_{II}(t)\\
&\quad+ \frac{\partial F_{II}}{\partial \eta}\left(t,\begin{pmatrix}
X_I(t)\\
X_{II}(t)
\end{pmatrix},\hat T_2\begin{pmatrix}
X_I\\
X_{II}
\end{pmatrix}(t) \right)\, \ddt \left(\hat T_2\begin{pmatrix}
X_I\\
X_{II}
\end{pmatrix}\right)(t),
\end{align*}
by which, using the first equation in~\eqref{syst:DAE.nonl.functional.closed.loop} and
\[
    \ddt \left(\hat T_2\begin{pmatrix}
X_I\\
X_{II}
\end{pmatrix}\right)(t) = z\left(S\begin{pmatrix}
X_I(t)\\
X_{II}(t)
\end{pmatrix},\hat T_3\begin{pmatrix}
X_I\\
X_{II}
\end{pmatrix}(t)\right),
\]
we obtain
\begin{multline*}
    \frac{\partial F_{II}}{\partial X_{II}}\left(t,\begin{pmatrix}
X_I(t)\\
X_{II}(t)
\end{pmatrix},\hat T_2\begin{pmatrix}
X_I\\
X_{II}
\end{pmatrix}(t) \right)\, \dot X_{II}(t) \\
= \hat F_{II}\left(t,\begin{pmatrix}
X_I(t)\\
X_{II}(t)
\end{pmatrix},T_1\begin{pmatrix}
X_I\\
X_{II}
\end{pmatrix}(t), \hat T_2\begin{pmatrix}
X_I\\
X_{II}
\end{pmatrix}(t),\hat T_3\begin{pmatrix}
X_I\\
X_{II}
\end{pmatrix}(t)  \right)
\end{multline*}
for some $\hat F_{II}:\tilde\cD\times\R^{3k}\to\R^{m-q}$.
We show that the matrix
\begin{multline}\label{eq:DAE.nonl.cM.def}
\frac{\partial F_{II}}{\partial X_{II}}(t,X_I,X_{II},\eta)=\tfrac{\partial f_2(X_I,X_{II})}{\partial X_{II}}
-\tfrac{\hat{k}f_4\left(d_5(t),\eta\right)}{1-\varphi_{II}(t)^2\|X_{II}-X_{\rmref,II}(t)\|^2}\cdot\\
\cdot\bigg(I_{m-q}+ \tfrac{2\varphi_{II}(t)^2(X_{II}-X_{\rmref,II}(t))(X_{II}-X_{\rmref,II}(t))^{\top}}{1-\varphi_{II}(t)^2\|X_{II}-X_{\rmref,II}(t)\|^2}\bigg)
\end{multline}
is invertible for all $(t,X_I,X_{II},\eta)\in\tilde \cD\times\R^k$:
The symmetry and positive semi-definiteness of
$$
\cG(t,X_{II}):=\tfrac{2\varphi_{II}(t)^2(X_{II}-X_{\rmref,II}(t))(X_{II}-X_{\rmref,II}(t))^{\top}}{1-\varphi_{II}(t)^2\|X_{II}-X_{\rmref,II}(t)\|^2}
$$
implies positive definiteness (and hence invertibility) of $I_{m-q}+\cG(t,X_{II})$ for all $(t,X_{II})\in\tilde \cD_{II}$, and by~\cite[Lem.~3.3]{BergIlch14} we further have
$$
\left\|\big(I_{m-q}+\cG(t,X_{II})\big)^{-1}\right\|\le 1.
$$
Therefore, according to  \eqref{DAE.nonl.gain.condition} and Assumption~\ref{def:DAE.nonl.syst.class}~(iv), we have for all $(t,X_I,X_{II},\eta)\in\tilde \cD\times\R^k$ that
\begin{multline*}
\left\|\big(1-\varphi_{II}(t)^2\|X_{II}-X_{\rmref,II}(t)\|^2\big)\hat{k}^{-1}\Big[f_4\left(d_5(t),\eta\right)\Big]^{-1} (I_{m-q}\!+\!\cG(t,X_{II}))^{-1} \tfrac{\partial f_2(X_I,X_{II})}{\partial X_{II}}\right\|\\
\le \hat{k}^{-1}\alpha^{-1}\!\left\|\tfrac{\partial f_2(X_I,X_{II})}{\partial X_{II}} \right\|\!\stackrel{\eqref{DAE.nonl.gain.condition}}{<}1.
\end{multline*}
This implies invertibility of $\frac{\partial F_{II}}{\partial X_{II}}(t,X_I,X_{II},\eta)$ for all $(t,X_I,X_{II},\eta)\in\tilde \cD\times\R^k$.
With $\tilde{F}_{II}: \tilde{\cD}\times\R^{3k}\to\R^{m-q}$ defined by
\[
    \tilde{F}_{II}(t,X_I,X_{II},\eta_1,\eta_2,\eta_3) := \left(\frac{\partial F_{II}}{\partial X_{II}}(t,X_I,X_{II},\eta_2)\right)^{-1}  \hat{F}_{II}(t,X_I,X_{II},\eta_1,\eta_2,\eta_3)
\]
and the first equation in~\eqref{syst:DAE.nonl.functional.closed.loop} we obtain the ODE
\begin{equation}\label{eq:DAE.nonl.ODE.closed.loop}
\begin{aligned}
\dot{X}_I(t)&=F_{I}\left(t,\begin{pmatrix}
X_I(t)\\
X_{II}(t)
\end{pmatrix},T_1\begin{pmatrix}
X_I\\
X_{II}
\end{pmatrix}(t) \right),\\
\dot{X}_{II}(t)&=\tilde{F}_{II}\left(t,\begin{pmatrix}
X_I(t)\\
X_{II}(t)
\end{pmatrix},T_1\begin{pmatrix}
X_I\\
X_{II}
\end{pmatrix}(t), \hat T_2\begin{pmatrix}
X_I\\
X_{II}
\end{pmatrix}(t),\hat T_3\begin{pmatrix}
X_I\\
X_{II}
\end{pmatrix}(t)  \right),
\end{aligned}
\end{equation}
with initial conditions~\eqref{syst:DAE.nonl.functional.closed.loop-IC}.\\
{\em Step 1f}:
Consider the initial value problem~\eqref{eq:DAE.nonl.ODE.closed.loop},~\eqref{syst:DAE.nonl.functional.closed.loop-IC}, then we have $(0,X_I(0),X_{II}(0))\in \tilde{\cD}$, $F_I$ is measurable in $t$, continuous in $(X_I,X_{II},\eta)$, and locally essentially bounded, and $\tilde{F}_{II}$ is measurable in $t$, continuous in $(X_I,X_{II},\eta_1,\eta_2,\eta_3)$, and locally essentially bounded. Therefore, by~\cite[Theorem B.1]{IlchRyan09}\footnote{In \cite{IlchRyan09} a domain $\cD\subseteq\R_{\ge 0}\times \R$ is considered, but the generalization to the higher dimensional case is straightforward.} we obtain existence of solutions to \eqref{eq:DAE.nonl.ODE.closed.loop}, and every solution can be extended to a maximal solution. Furthermore, for a maximal solution $(X_I,X_{II}):[-h,\omega)\to \R^{\overline{r}+m-q}$, $\omega\in (0,\infty]$, of~\eqref{eq:DAE.nonl.ODE.closed.loop},~\eqref{syst:DAE.nonl.functional.closed.loop-IC} the closure of the graph of this solution is not a compact subset of $\tilde{\cD}$.

We show that $(X_I,X_{II})$ is also a maximal solution of~\eqref{syst:DAE.nonl.functional.closed.loop}. Since $(X_I,X_{II})$ is particular satisfies, by construction,
\[
\forall\, t\in[0,\omega):\ 0 = \ddt F_{II}\left(t,\begin{pmatrix}
X_I(t)\\
X_{II}(t)
\end{pmatrix},\hat T_2\begin{pmatrix}
X_I\\
X_{II}
\end{pmatrix}(t) \right),
\]
there exists $c\in\R^{m-q}$ such that
\[
\forall\, t\in[0,\omega):\ c = F_{II}\left(t,\begin{pmatrix}
X_I(t)\\
X_{II}(t)
\end{pmatrix},\hat T_2\begin{pmatrix}
X_I\\
X_{II}
\end{pmatrix}(t) \right).
\]
Invoking~\eqref{syst:DAE.nonl.functional.closed.loop-IC}, the definition of~$F_{II}$ and $\hat T_2$, and the consistency condition~\eqref{eq:consistent-IV} we may infer that $c=0$. Therefore, $(X_I,X_{II})$ is a solution of~\eqref{syst:DAE.nonl.functional.closed.loop}. Furthermore, $(X_I,X_{II})$ is also a maximal solution of~\eqref{syst:DAE.nonl.functional.closed.loop}, since any right extension would be a solution of~\eqref{eq:DAE.nonl.ODE.closed.loop} following the procedure in Step~1e, a contradiction.


Recall $X_{I}(t)= \Big(y_1(t),\dots,y_1^{(r_1-1)}(t)\dots, y_q^{(r_q-1)}(t)\Big)^{\top}$, $X_{II}(t)=(y_{q+1}(t),\dots,y_m(t))^{\top}$ and define
\begin{equation}(e_{10},\dots,e_{1,r_1-1},\dots,e_{q,r_q-1},e_{q+1},\dots,e_m):[0,\omega)\to \R^{\overline{r}+m-q}
\label{eq:efun}\end{equation}
by
\begin{align*}
e_{ij}(t)&=\tilde{K}_{ij}(t,y_i(t),\dots,y_i^{(j)}(t)),&& \text{for } i=1,\dots,q \text{ and } j=0,\dots, r_i-1,\\
e_i(t)&=y_i(t)-y_{\rmref,i}(t),&& \text{for } i=q+1,\dots,m,
\end{align*}
then the closure of the graph of the function in \eqref{eq:efun}
is not a compact subset of $\cD$.\\
{\em Step 2}: We show boundedness of the gain functions $k_{I}(\cdot)$, $k_{II}(\cdot)$ and $k_{ij}(\cdot)$ as in~\eqref{DAE.nonl.funn.controller} on $[0,\omega)$. This also proves~\eqref{DAE.nonl.bounded.away.error}.\\
{\em Step 2a}: The proof of boundedness of $k_{ij}(\cdot)$ for $i=1,\dots,q$, $j=0,\dots,r_i-2$ on $[0,\omega)$ is analogous to {Step 2a} of the proof of \cite[Thm.~3.1]{BergLe18a} and hence omitted.\\
{\em Step 2b}: We prove by induction that there exist constants ${M}^\ell_{ij},{N}^\ell_{ij},{K}^\ell_{ij}>0$ such that, for all $t\in [0,\omega)$,
\begin{equation}\label{DAE.nonl.der.bound}
\left|\left(\ddts\right)^\ell \left[k_{ij}(t)e_{ij}(t)\right]\right|\le {M}^\ell_{ij},\quad \left|\left(\ddts\right)^\ell e_{ij}(t)\right|\le {N}^\ell_{ij},\quad \left|\left(\ddts\right)^\ell k_{ij}(t)\right|\le {K}^\ell_{ij},
\end{equation}
for $i=1,\dots,q$, $j=0,\dots, r_i-2$, and $\ell=0,\dots, r_i-1-j$.\\
First, we may infer from {\it Step 2a} that $k_{ij}(\cdot)$, for $i=1,\dots,q$, $j=0,\dots, r_i-2$, are bounded. Furthermore, $e_{ij}$ are bounded since they evolve in the respective performance funnels. Therefore, for each $i=1,\dots,q$ and $j=0,\ldots,r_i-2$, \eqref{DAE.nonl.der.bound} is true whenever $\ell=0$. Fix $i\in\{1,\ldots,q\}$. We prove~\eqref{DAE.nonl.der.bound} for $j=r_i-2$ and $\ell=1$:
\begin{align*}
  \dot e_{i,r_i-2}(t) &= e_{i,r_i-1}(t)-k_{i,r_i-2}(t)e_{i,r_i-2}(t),\\
  \dot k_{i,r_i-2}(t) &= 2 k^2_{i,r_i-2}(t)\big(\varphi_{i,r_i-2}^2(t) e_{i,r_i-2}(t) \dot e_{i,r_i-2}(t)\\
  &\quad+ \varphi_{i,r_i-2}(t)\dot{\varphi}_{i,r_i-2}(t) |e_{i,r_i-2}(t)|^2\big),\\
  \ddts \left[k_{i,r_i-2}(t)e_{i,r_i-2}(t)\right]&= \dot{k}_{i,r_i-2}(t) e_{i,r_i-2}(t)+k_{i,r_i-2}(t)\dot{e}_{i,r_i-2}(t).
\end{align*}
Boundedness of $k_{i,r_i-2}$, $\varphi_{i,r_i-2}$, $\dot{\varphi}_{i,r_i-2}$, $e_{i,r_i-2}$ together with the above equations implies that $\dot e_{i,r_i-2}(t)$, $\dot k_{i,r_i-2}(t)$ and $\ddts \left[k_{i,r_i-2}(t)e_{i,r_i-2}(t)\right]$ are bounded. Now consider indices $s\in\{0,\ldots, r_i-3\}$ and $l\in\{0,\ldots,r_i-1-s\}$ and assume that~\eqref{DAE.nonl.der.bound} is true for all $j = s+1,\ldots,r_i-2$ and all $\ell=0,\ldots,r_i-1-j$ as well as for $j=s$ and all $\ell=0,\ldots,l-1$. We show that it is true for $j=s$ and $\ell=l$:
\begin{align*}
  \left(\ddts\right)^l e_{is}(t)&= \left(\ddts\right)^{l-1}\left[e_{i,s+1}(t)-k_{is}(t)e_{is}(t)\right]\\&= \left(\ddts\right)^{l-1} e_{i,s+1}(t) - \left(\ddts\right)^{l-1} \left[k_{is}(t)e_{is}(t)\right],\\
  \left(\ddts\right)^l k_{is}(t)&= \left(\ddts\right)^{l-1}\!\! \Big( 2k^2_{is}(t) \big(\varphi_{is}^2(t) e_{is}(t) \dot{e}_{is}(t)\! + \!\varphi_{is}(t)\dot{\varphi}_{is}(t) |e_{is}(t)|^2\big)\Big),\\
\left(\ddts\right)^l \left[k_{is}(t)e_{is}(t)\right]&= \left(\ddts\right)^{l-1} \big( \dot{k}_{is}(t) e_{is}(t)+k_{is}(t)\dot{e}_{is}(t)\big).
\end{align*}
Then, successive application of the product rule and using the induction hypothesis as wells as the fact that $\varphi_{is},\dot{\varphi}_{is},\ldots,\varphi_{is}^{(r_i-s)}$ are bounded, yields that the above terms are bounded. Therefore, the proof of \eqref{DAE.nonl.der.bound} is complete.\\
It follows from \eqref{DAE.nonl.der.bound} and \eqref{eq:DAE.nonl.eiqrel} that, for all $i=1,\dots,q$ and $j=0,\dots, r_i-1$, $e_i^{(j)}$ is bounded on $[0,\omega)$.\\
{\em Step 2c}: We show that $k_{I}(\cdot)$ as in~\eqref{DAE.nonl.funn.controller} is bounded. It follows from~\eqref{eq:DAE.nonl.eiqrel} that, for $i=1,\dots,q$,
$$
e_i^{(r_i)}(t)=\dot{e}_{i,r_i-1}(t)-\sum\limits_{j=0}^{r_i-2}\left(\ddts\right)^{r_i-1-j} \left[k_{ij}(t)e_{ij}(t)\right].
$$
Then we find that by~\eqref{syst:DAE.nonl.functional.closed.loop}
\begin{equation}\label{eq:DAE.nonl.eI}
\begin{aligned}
\dot{e}_I(t)&=f_1\Big(d_1(t), T_1\big(y_1,\dots,y_1^{(r_1-1)},\ldots,y_q^{(r_q-1)},y_{q+1},\ldots,y_m\big)(t)\Big)\\
&\quad -\Gamma_I\Big(d_1(t), T_1\big(y_1,\dots,y_1^{(r_1-1)},\ldots,y_q^{(r_q-1)},y_{q+1},\ldots,y_m\big)(t)\Big)\, k_I(t)e_I(t)\\
&\quad +\begin{pmatrix}
\sum\limits_{j=0}^{r_1-2}\left(\ddts\right)^{r_1-1-j}k_{1j}(t)e_{1j}(t)\\
\sum\limits_{j=0}^{r_2-2}\left(\ddts\right)^{r_2-1-j}k_{2j}(t)e_{2j}(t)\\
\vdots\\
\sum\limits_{j=0}^{r_q-2}\left(\ddts\right)^{r_q-1-j}k_{qj}(t)e_{qj}(t)\\
\end{pmatrix}
-\begin{pmatrix}
y_{\rmref,1}^{(r_1)}(t)\\
\vdots\\
y_{\rmref,q}^{(r_q)}(t)\\
\end{pmatrix}.
\end{aligned}
\end{equation}
Again we use  $X_{I}(t)= \Big(y_1(t),\dots,y_1^{(r_1-1)}(t)\dots, y_q^{(r_q-1)}(t)\Big)^{\top}$, $X_{II}(t)=(y_{q+1}(t),\dots,y_m(t))^{\top}$ and we set, for $t\in[0,\omega)$,
\begin{equation}\label{eq:DAE.nonl.hatFI}
\begin{aligned}
\hat{F}_I(t)&:=f_1\Big(d_1(t), T_1\big(X_{I},X_{II}\big)(t)\Big)\\
&\quad +\begin{pmatrix}
\sum\limits_{j=0}^{r_1-2}\left(\ddts\right)^{r_1-1-j}k_{1j}(t)e_{1j}(t)\\
\sum\limits_{j=0}^{r_2-2}\left(\ddts\right)^{r_2-1-j}k_{2j}(t)e_{2j}(t)\\
\vdots\\
\sum\limits_{j=0}^{r_q-2}\left(\ddts\right)^{r_q-1-j}k_{qj}(t)e_{qj}(t)\\
\end{pmatrix}
-\begin{pmatrix}
y_{\rmref,1}^{(r_1)}(t)\\
\vdots\\
y_{\rmref,q}^{(r_q)}(t)\\
\end{pmatrix}.
\end{aligned}
\end{equation}
We obtain from \eqref{DAE.nonl.der.bound} and~\eqref{eq:DAE.nonl.eiqrel} that~$e_i^{(j)}$ is bounded on the interval $[0,\omega)$ for $i=1,\ldots,q$ and $j=0,\ldots,r_i-2$. Furthermore, $e_I$ evolves in the performance funnel $\cF_{\varphi_I}^q$, thus $|e_{i,r_i-1}(t)|^2 \le \|e_I(t)\|^2 < \varphi_I(t)^{-1}$ for all $t\in [0,\omega)$, so $e_{i,r_i-1}$ is bounded on $[0,\omega)$ for $i=1,\ldots,q$. Invoking boundedness of $y_{\rmref,i}^{(j)}$ yields boundedness of $y_i^{(j)}$ for $i=1,\ldots,q$, $j=0,\ldots,r_i-1$. Then the bounded-input, bounded-output property of $T_1$ in Definition~\ref{def:DAE.nonl.oper.T.clas}~(iii) implies that $T_1\big(X_{I},X_{II}\big)$ is bounded by
\[
    M_{T_1}:=\|T_1\big(X_{I},X_{II}\big)|_{[0,\omega)}\|_{\infty}.
\]
This property together with~\eqref{DAE.nonl.der.bound}, continuity of $f_1$ and boundedness of $d_1$ yields that $\hat{F}_I(\cdot)$ is bounded on $[0,\omega)$. In other words, there exists some $M_{\hat{F}_I}>0$ such that $\|\hat{F}_I|_{[0,\omega)}\|_\infty\le M_{\hat{F}_I}$. Now define the compact set
$$
\Omega=\setdef{\!(\delta,\eta,e_I)\in \R^s\times\R^k\times\R^q\!}{\!\|\delta\|\le \|d_2|_{[0,\omega)}\|_{\infty},\ \|\eta\|\le M_{T_1},\ \|e_I\|=1\!\!},
$$
then, since $\Gamma_I+\Gamma_I^\top$ is pointwise positive definite by Assumption~\ref{def:DAE.nonl.syst.class}\,(i) and the map
$$
\Omega\ni (\delta,\eta,e_I) \mapsto e^{\top}_I\big(\Gamma_I(\delta,\eta)+\Gamma_I(\delta,\eta)^\top\big)e_I\in \R_{> 0}
$$
is continuous, it follows that there exists $\gamma>0$ such that
$$
\forall\, (\delta,\eta,e_I)\in \Omega:\quad e^{\top}_I\big(\Gamma_I(\delta,\eta)+\Gamma_I(\delta,\eta)^\top\big)e_I\ge \gamma.
$$
Therefore, we have for all $t\in [0,\omega)$ that
$$
 e_I(t)^{\top}\Big(\Gamma_I\Big(d_1(t), T_1\big(X_{I},X_{II}\big)(t)\Big) + \Gamma_I\Big(d_1(t), T_1\big(X_{I},X_{II}\big)(t)\Big)^\top\Big)e_I(t)\ge \gamma\|e_I(t)\|^2.
$$
Now, set $\psi_I(t):=\varphi_I(t)^{-1}$ for $t\in(0,\omega)$, let $T_I\in(0,\omega)$ be arbitrary but fixed and set $\lambda_I := \inf_{t\in(0,\omega)} \psi_I(t)$. Since $\dot \varphi_I$ is bounded and $\liminf_{t\to\infty}\varphi_I(t) >0$ we find that $\ddts \psi_I|_{[0,\omega)}$ is bounded and hence $\psi_I|_{[0,\omega)}$ is Lipschitz continuous with Lipschitz bound $L_I>0$. Choose $\varepsilon_{I}>0$ small enough  such that
\begin{align}
\varepsilon_{I} &\le \min\left\{ \frac{\lambda_{I}}{2}, \inf\limits_{t\in(0,T_{I}]} (\psi_{I}(t) - \|e_{I}(t)\|)\right\}\nonumber\\
\text{and }\quad
\label{eq:DAE.nonl.funn.LI}
    L_{I} &\le \dfrac{\lambda_{I}^2}{8\varepsilon_{I}}\gamma-M_{\hat{F}_I},
\end{align}
We show that
\begin{equation}\label{DAE.nonl.eI.estimate}
\forall\, t\in (0,\omega):\ \psi_{I}(t)-\|e_{I}(t)\|\ge \varepsilon_{I}.
\end{equation}
By definition of $\varepsilon_{I}$ this holds on $(0,T_{I}]$. Seeking a contradiction suppose that
$$
    \exists\, t_{I,1}\in [T_I,\omega):\ \psi_{I}(t_{I,1})-\|e_{I}(t_{I,1})\|<\varepsilon_{I}.
$$
Set $t_{I,0}=\max\{t\in [T_{I},t_{I,1})\ |\ \psi_{I}(t)-\|e_{I}(t)\|= \varepsilon_{I}\}$. Then, for all $t\in [t_{I,0},t_{I,1}]$, we have
\begin{align*}
  \psi_{I}(t)-\|e_{I}(t)\| &\le \varepsilon_{I},\\
  \|e_{I}(t)\|\ge \psi_{I}(t)-\varepsilon_{I} &\ge \dfrac{\lambda_{I}}{2},\\
  k_{I}(t)=\dfrac{1}{1-\varphi_{I}^2(t)\|e_{I}(t)\|^2}&\ge \dfrac{\lambda_{I}}{2\varepsilon_{I}}.
\end{align*}
Then it follows from~\eqref{eq:DAE.nonl.eI} and~\eqref{eq:DAE.nonl.hatFI} that for all $t\in [t_{I,0},t_{I,1}]$,
\begin{align*}
&\dfrac{1}{2}\ddt\|e_{I}(t)\|^2=\dfrac{1}{2}\left(e_{I}^{\top}(t)\dot{e}_{I}(t) + \dot e_{I}^{\top}(t){e}_{I}(t)\right)\\
&=e_{I}^{\top}(t)\left(\hat{F}_I(t)-\dfrac{1}{2} \Big(\Gamma_I\Big(d_1(t), T_1\big(X_{I},X_{II}\big)(t)\Big) + \Gamma_I\Big(d_1(t), T_1\big(X_{I},X_{II}\big)(t)\Big)^\top\Big) k_I(t)e_I(t) \right)\\
&\le\left(M_{\hat{F}_I}-\dfrac{\lambda_{I}^2}{8\varepsilon_{I}}\gamma\right)\|e_{I}(t)\|\overset{\eqref{eq:DAE.nonl.funn.LI}}{\le}-L_{I}\|e_{I}(t)\|.
\end{align*}
Then, using $\|e_I(t)\|\ge \frac{\lambda_I}{2} > 0$ for all $t\in [t_{I,0},t_{I,1}]$,
\begin{align*}
\|e_{I}(t_{I,1})\|-\|e_{I}(t_{I,0})\|&=\int\limits_{t_{I,0}}^{t_{I,1}}\frac{1}{2}\|e_{I}(t)\|^{-1} \ddts\|e_{I}(t)\|^2\, \ds{t} \\ &\le-L_{I}(t_{I,1}-t_{I,0})\le-|\psi_I(t_{I,1})-\psi_I(t_{I,0})|\le\psi_I(t_{I,1})-\psi_I(t_{I,0}),
\end{align*}
and thus we obtain $\varepsilon_{I}=\psi_{I}(t_{I,0})-\|e_{I}(t_{I,0})\|\le\psi_{I}(t_{I,1})-\|e_{I}(t_{I,1})\|<\varepsilon_{I}$, a contradiction.\\
{\em Step 2d}: We show that $k_{II}(\cdot)$ as in~\eqref{DAE.nonl.funn.controller} is bounded. Seeking a contradiction, we assume that $k_{II}(t)\to \infty$ for $t\to \omega$. Set, for $t\in[0,\omega)$,
\begin{equation}\label{eq:DAE.nonl.FII}
\begin{aligned}
\check{F}_{II}(t)&:=f_2\big(X_I(t),X_{II}(t)\big)+f_3\big(d_3(t),(T_2y)(t)\big) - \Gamma_{II}\big(d_4(t),(T_2y)(t)\big)\, k_I(t) e_I(t).
\end{aligned}\end{equation}
Since~$k_I$ is bounded on $[0\omega)$ by Step~2c, it follows from {Step 2b}, boundedness of $T_2(y)$, $d_3$ and $d_4$ and continuity of $f_2$, $f_3$ and $\Gamma_{II}$ that $\check{F}_{II}(\cdot)$ is bounded on $[0,\omega)$. By~\eqref{syst:DAE.nonl.functional.closed.loop}
we have
\begin{equation}\label{eq:DAE.nonl.eII}
0=\check{F}_{II}(t)-f_4\big(d_5(t),(T_2y)(t)\big)\,k_{II}(t)e_{II}(t).
\end{equation}
We show that $e_{II}(t)\to 0$ for $t\to \omega$. Seeking a contradiction, assume that there exist $\kappa>0$ and a sequence $(t_n)\subset \R_{\ge 0}$ with $t_n\nearrow \omega$ such that $\|e_{II}(t_n)\|\ge \kappa$ for all $n\in \N$. Then, from~\eqref{eq:DAE.nonl.eII} we obtain, for all $t\in[0,\omega)$,
$$
\|\check{F}_{II}(t)\|=\|f_4\big(d_5(t),(T_2y)(t)\big)\,k_{II}(t)e_{II}(t)\| = |f_4\big(d_5(t),(T_2y)(t)\big)| \cdot |k_{II}(t)| \cdot \|e_{II}(t)\|.
$$
Since $k_{II}(t)\to \infty$ for $t\to \omega$, $\|e_{II}(t_n)\|\ge \kappa$ and $f_4\big(d_5(t_n),(T_2y)(t_n)\big)\ge \alpha$, we find that
$$
\|\check{F}_{II}(t_n)\|\ge \alpha\,\kappa\,k_{II}(t_n)\to \infty\ \mbox{ for } n\to \infty,
$$
which contradicts boundedness of $\check{F}_{II}(\cdot)$.

Hence, we have $e_{II}(t)\to 0$ for $t\to \omega$, by which $\lim_{t\to \infty}\varphi_{II}(t)^2\|e_{II}(t)\|^2=0$ because $\varphi_{II}(\cdot)$ is bounded. This leads to the contradiction $\lim_{t\to \infty}k_{II}(t)=\hat{k}$, thus $k_{II}(\cdot)$ is bounded.\\
{\em Step 3}: We show that $\omega=\infty$. Seeking a~contradiction, we assume that $\omega<\infty$. Then, since $e_I, e_{II}, k_I, k_{II}$ and $e_{ij}, k_{ij}$ are bounded for $i=1,\dots,q$, $j=0,\dots,r_i-2$ by {Step 2}, it follows that the closure of the graph of the function in~\eqref{eq:efun} is a compact subset of $\cD$, which is a contradiction. This finishes the proof of the theorem.
\end{proof}
\section{Simulations}\label{sect:DAE.nonl.sims}

In this section we illustrate the application of the funnel controller~\eqref{DAE.nonl.funn.controller} by considering the following academic example:
\begin{equation}\label{DAE.nonl.sims.exam20}
\begin{aligned}
\ddot{y}_1(t)=& -\sin y_1(t)+y_1(t)\dot{y}_1(t)+y_2(t)^2\\&\quad+\dot y_1(t)^2 T(y_1,y_2)(t)+(y_1(t)^2+y_2(t)^4+1)u_I(t),\\
0=&\,y_1(t)^3+y_1(t)\dot{y}_1(t)^3+ y_2(t)+ T(y_1,y_2)(t)+\\&\quad + T(y_1,y_2)(t) u_I(t)+u_{II}(t),
\end{aligned}
\end{equation}
where $T:C(\R_{\ge 0}\to\R^m)\to C^1(\R_{\ge 0}\to\R)$ is given by
$$
T(y_1,y_2)(t):=e^{-2t}\eta^0+\int\limits_{0}^te^{-2(t-s)}\Big(2y_1(s)-y_2(s)\Big)ds, \quad t\ge 0,
$$
for any fix $\eta^0\in \R$. Similar as we have calculated for the operator~$T_2$ on page~\pageref{pp:operatorT2}, we may calculate that $T\in \T_{2,1,0}^{\rm DAE}$. Define
\[
    T_1(y_1,y_1^d,y_2)(t):= \begin{pmatrix} y_1(t)\\ y_1^d(t)\\ y_2(t) \\ T(y_1,y_2)(t)\end{pmatrix},\quad t\ge 0,
\]
then $T_1 \in \T_{2,3,0}$, and set $T_2 := T$. Furthermore, define the functions
\begin{equation*}
\begin{aligned}
f_1:\R^4\to\R,\ & (\eta_1,\eta_2,\eta_3,\eta_4) \mapsto -\sin \eta_1 + \eta_3 \eta_2 + \eta_3^2 + \eta_2^2 \eta_4,\\
\Gamma_I:\R^4\to\R,\ & (\eta_1,\eta_2,\eta_3,\eta_4) \mapsto \eta_1^2+\eta_3^4+1,\\
f_2:\R^4\to\R,\ & (y_1,y_1^d,y_2) \mapsto y_1^3+ y_1 (y_1^d)^3 + y_2,\\
f_3:\R\to\R,\ & \eta \mapsto \eta,\\
\Gamma_{II}:\R\to\R,\ & \eta \mapsto \eta,\\
f_4:\R\to\R,\ & \eta \mapsto 1.
\end{aligned}
\end{equation*}
Then system~\eqref{DAE.nonl.sims.exam20} is of the form~\eqref{syst:DAE.nonl} with $m=2$, $q=1$ and $r_1=2$. It is straightforward to check that Assumption~\ref{def:DAE.nonl.syst.class} is satisfied. In particular, condition~(iii) is satisfied, because
\[
    f_2'(y_1,y_1^d,y_2) \begin{bmatrix} 0\\ I_{m-q}\end{bmatrix} = \frac{\partial f_2}{\partial y_2}(y_1,y_1^d,y_2) = 1
\]
is bounded. Furthermore, $f_4(\eta) \ge 1 =: \alpha$ for all $\eta\in\R$, and hence we may choose $\hat k = 2$, with which condition~\eqref{DAE.nonl.gain.condition} is satisfied.

For the simulation we choose the reference signal $y_{\rmref}(t) =(\cos 2t, \sin t)^{\top}$, and initial values
\[
    y_1(0) = \dot y_1(0) = y_2(0) = 0\quad \text{and}\quad \eta^0 = 0.
\]
For the controller~\eqref{DAE.nonl.funn.controller} we choose the funnel functions $\varphi_{10} = \varphi_{I} = \varphi_{II} = \varphi$ with
\begin{equation*}
\varphi:\R_{\ge 0}\to \R_{\ge 0},\ t\mapsto \tfrac{1}{2}te^{-t}+2\arctan t.
\end{equation*}
It is straightforward to check that $\dot \varphi$ and $\ddot \varphi$ are bounded, thus $\varphi\in\Phi_2$. Moreover, since $\varphi(0)=0$, no restriction is put on the initial error and we find that~\eqref{DAE.nonl.init.condition} is satisfied and $k_{10}(0) = k_I(0) = 1$ and $k_{II}(0) = 2$. Furthermore,
\begin{align*}
    e_I(0) &= e_{1,1}(0) = \dot e_{10}(0) + k_{10}(0) e_{10}(0) = \dot y_1(0) - \dot y_{\rmref,1}(0) + k_{10}(0) \big( y_1(0) - y_{\rmref,1}(0)\big) = -1,\\
     e_{II}(0) &= y_2(0) - y_{\rmref,2}(0) = 0,
\end{align*}
and hence we obtain
\begin{align*}
    u_I(0) = -k_I(0) e_I(0) = 1,\quad
     u_{II}(0) = -k_{II}(0) e_{II}(0) = 0.
\end{align*}
Since $h=0$, we find that in view of Remark~\ref{Rem:cons-init-val} the localization of~$T_2$ satisfies $T_2(0,0) = 0$. With this finally find that the initial value is indeed consistent, i.e., condition~\eqref{eq:consistent-IV} is satisfied. We have now verified all assumptions of Theorem~\ref{Thm:DAE.nonl.funn.controller}, by which funnel control via~\eqref{DAE.nonl.funn.controller} is feasible for the system~\eqref{DAE.nonl.sims.exam20}

The simulation of the controller~\eqref{DAE.nonl.funn.controller} applied to~\eqref{DAE.nonl.sims.exam20} has been performed in MATLAB (solver: \textsf{ode15s}, rel. tol.: $10^{-14}$, abs. tol.: $10^{-10}$) over the time interval [0,10] and is depicted in Figure~\ref{fig:DAE.nonl.sims.exam20}.

\captionsetup[subfloat]{labelformat=empty}
\begin{figure}[h!tb]
  \centering
  \subfloat[Fig. \ref{fig:DAE.nonl.sims.exam20}a: Funnel and tracking errors]
{
\centering
  \includegraphics[width=10cm]{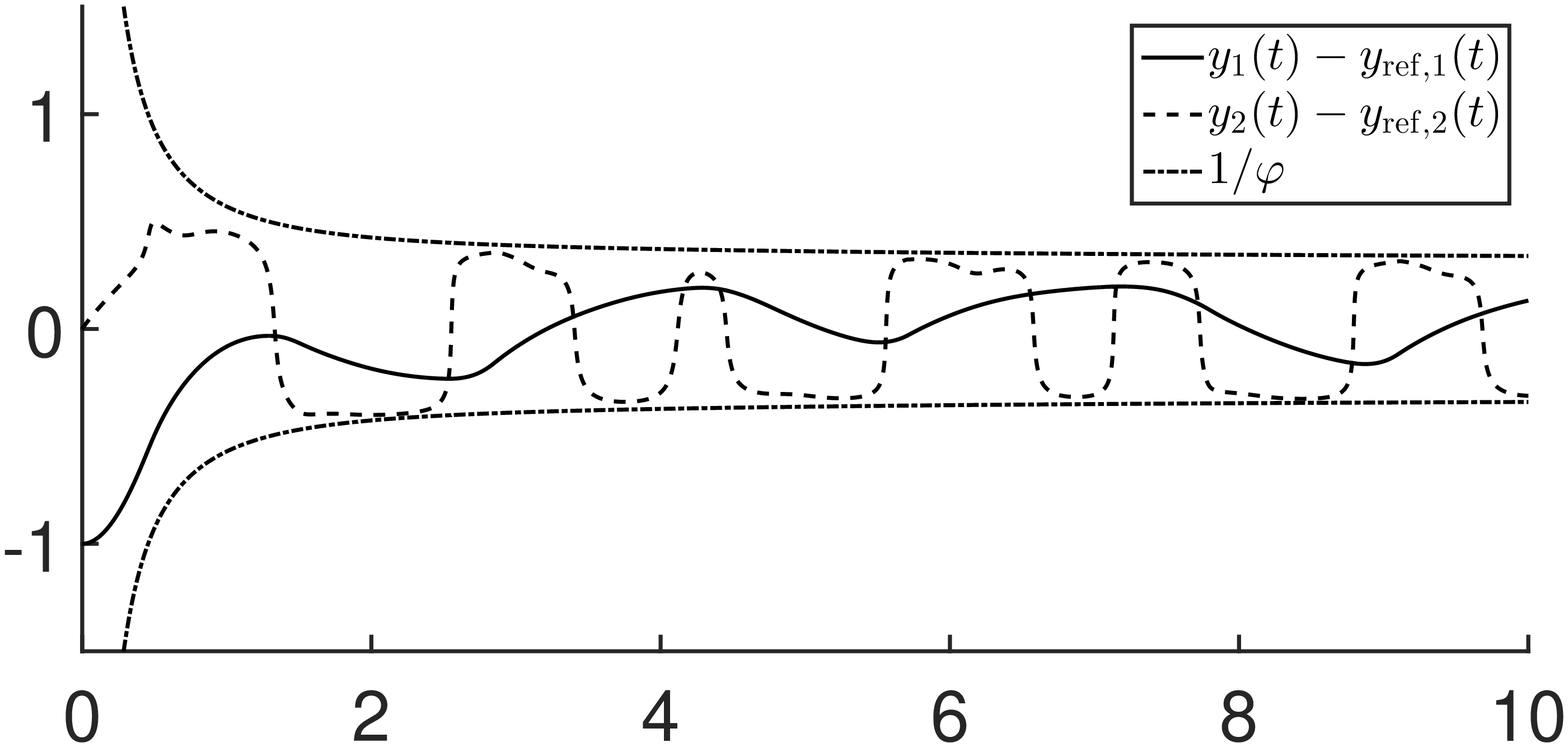}
\label{fig:DAE.nonl.sims.exam20.e}
}\\
\subfloat[Fig. \ref{fig:DAE.nonl.sims.exam20}b: Input functions]
{
\centering
  \includegraphics[width=10cm]{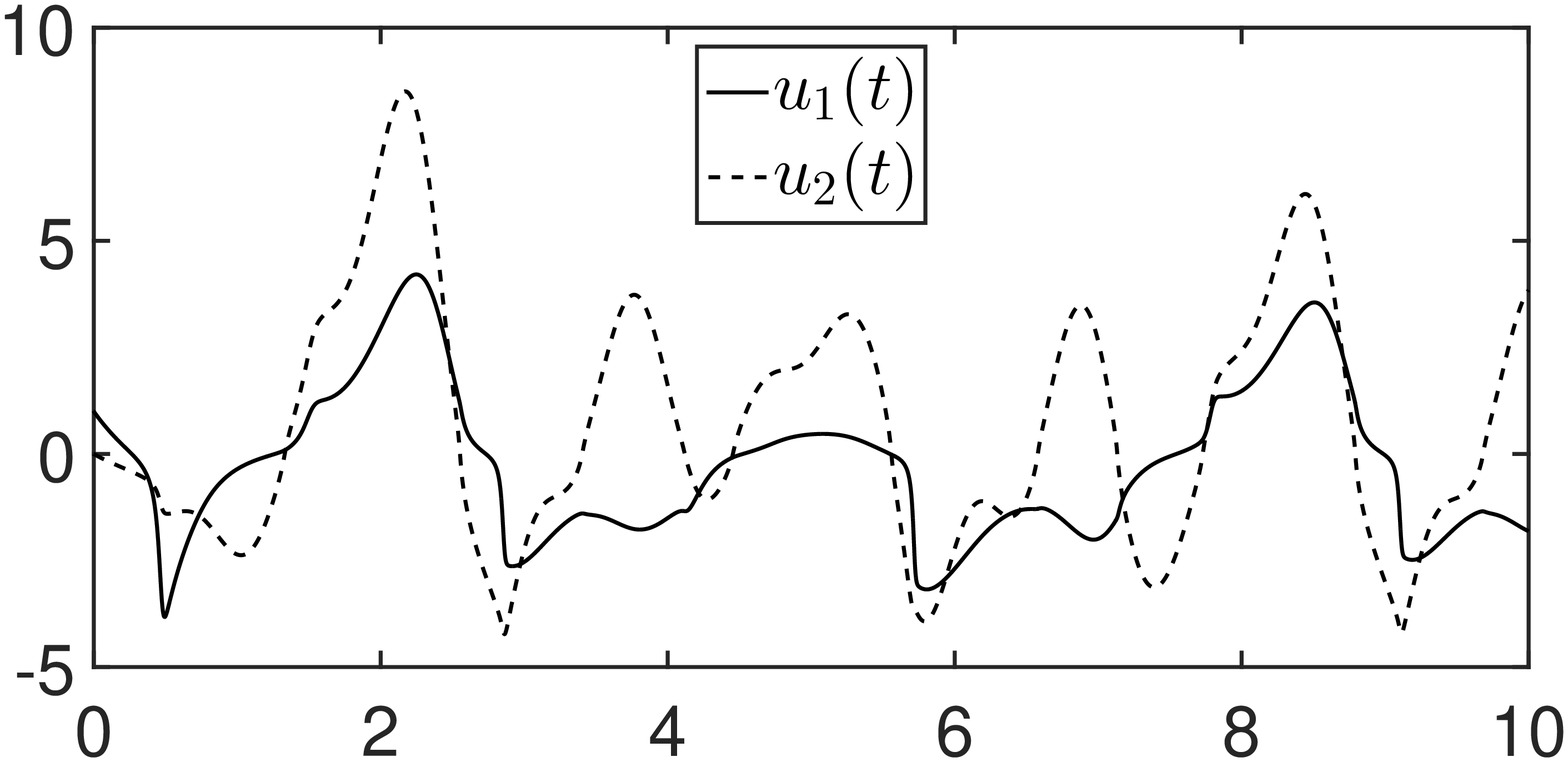}
\label{fig:DAE.nonl.sims.exam20.u}
}
\caption{Simulation of the controller \eqref{DAE.nonl.funn.controller} for the system \eqref{DAE.nonl.sims.exam20}.}
\label{fig:DAE.nonl.sims.exam20}
\end{figure}

Figure \ref{fig:DAE.nonl.sims.exam20.e} shows the tracking error components, which stay uniformly within the funnel  boundaries. The components of the generated input functions are shown in Figure~\ref{fig:DAE.nonl.sims.exam20.u}, which exhibit an acceptable performance.
\bibliographystyle{spmpsci}      

\begin{thebibliography}{10}
\providecommand{\url}[1]{{#1}}
\providecommand{\urlprefix}{URL }
\expandafter\ifx\csname urlstyle\endcsname\relax
  \providecommand{\doi}[1]{DOI~\discretionary{}{}{}#1}\else
  \providecommand{\doi}{DOI~\discretionary{}{}{}\begingroup
  \urlstyle{rm}\Url}\fi

\bibitem{Berg14a}
Berger, T.: On differential-algebraic control systems.
\newblock Ph.D. thesis, Institut f{\"u}r Mathematik, Technische Universit{\"a}t
  Ilmenau, Universit{\"a}tsverlag Ilmenau, Germany (2014).
\newblock
  \urlprefix\url{http://www.db-thueringen.de/servlets/DocumentServlet?id=22652}

\bibitem{Berg14c}
Berger, T.: Zero dynamics and stabilization for linear {DAE}s.
\newblock In: S.~Sch\"ops, A.~Bartel, M.~G\"unther, E.J.W. ter Maten, P.C.
  M\"uller (eds.) Progress in Differential-Algebraic Equations,
  Differential-Algebraic Equations Forum, pp. 21--45. Springer-Verlag,
  Berlin-Heidelberg (2014)

\bibitem{Berg16b}
Berger, T.: Zero dynamics and funnel control of general linear
  differential-algebraic systems.
\newblock {ESAIM} Control Optim. Calc. Var. \textbf{22}(2), 371--403 (2016)

\bibitem{BergIlch12b}
Berger, T., Ilchmann, A., Reis, T.: Zero dynamics and funnel control of linear
  differential-algebraic systems.
\newblock Math. Control Signals Syst. \textbf{24}(3), 219--263 (2012)

\bibitem{BergIlch14}
Berger, T., Ilchmann, A., Reis, T.: Funnel control for nonlinear functional
  differential-algebraic systems.
\newblock In: Proceedings of the MTNS 2014, pp. 46--53. Groningen, NL (2014)

\bibitem{BergIlch12a}
Berger, T., Ilchmann, A., Trenn, S.: The quasi-{W}eierstra{\ss} form for
  regular matrix pencils.
\newblock Linear Algebra Appl. \textbf{436}(10), 4052--4069 (2012).
\newblock \doi{10.1016/j.laa.2009.12.036}

\bibitem{BergLe18a}
Berger, T., L{\^e}, H.H., Reis, T.: Funnel control for nonlinear systems with
  known strict relative degree.
\newblock Automatica \textbf{87}, 345--357 (2018).
\newblock \doi{10.1016/j.automatica.2017.10.017}

\bibitem{BergOtto19}
Berger, T., Otto, S., Reis, T., Seifried, R.: Combined open-loop and funnel
  control for underactuated multibody systems.
\newblock Nonlinear Dynamics \textbf{95}, 1977--1998 (2019).
\newblock \doi{10.1007/s11071-018-4672-5}

\bibitem{BergRaue18}
Berger, T., Rauert, A.L.: A universal model-free and safe adaptive cruise
  control mechanism.
\newblock In: Proceedings of the MTNS 2018, pp. 925--932. Hong Kong (2018)

\bibitem{BergReis14a}
Berger, T., Reis, T.: Zero dynamics and funnel control for linear electrical
  circuits.
\newblock J. Franklin Inst. \textbf{351}(11), 5099--5132 (2014)

\bibitem{ByrnIsid84}
Byrnes, C.I., Isidori, A.: A frequency domain philosophy for nonlinear systems,
  with application to stabilization and to adaptive control.
\newblock In: Proc. 23rd~{IEEE} Conf. Decis. Control, vol.~1, pp. 1569--1573
  (1984)

\bibitem{ByrnIsid85}
Byrnes, C.I., Isidori, A.: Global feedback stabilization of nonlinear systems.
\newblock In: Proc. 24th~{IEEE} Conf. Decis. Control, Ft. Lauderdale, FL, 1,
  pp. 1031--1037 (1985)

\bibitem{ByrnIsid88}
Byrnes, C.I., Isidori, A.: Local stabilization of minimum-phase nonlinear
  systems.
\newblock Syst. Control Lett. \textbf{11}(1), 9--17 (1988)

\bibitem{ByrnIsid89}
Byrnes, C.I., Isidori, A.: New results and examples in nonlinear feedback
  stabilization.
\newblock Syst. Control Lett. \textbf{12}(5), 437--442 (1989)

\bibitem{FuhrHelm15}
Fuhrmann, P.A., Helmke, U.: The Mathematics of Networks of Linear Systems.
\newblock Springer-Verlag, New York, NY (2015)

\bibitem{Hack17}
Hackl, C.M.: Non-identifier Based Adaptive Control in Mechatronics--Theory and
  Application, \emph{Lecture Notes in Control and Information Sciences}, vol.
  466.
\newblock Springer-Verlag, Cham, Switzerland (2017)

\bibitem{HackHopf13}
Hackl, C.M., Hopfe, N., Ilchmann, A., Mueller, M., Trenn, S.: Funnel control
  for systems with relative degree two.
\newblock {SIAM} J. Control Optim. \textbf{51}(2), 965--995 (2013)

\bibitem{Ilch13}
Ilchmann, A.: Decentralized tracking of interconnected systems.
\newblock In: K.~H\"{u}per, J.~Trumpf (eds.) Mathematical System Theory -
  Festschrift in Honor of Uwe Helmke on the Occasion of his Sixtieth Birthday,
  pp. 229--245. CreateSpace (2013)

\bibitem{IlchRyan94}
Ilchmann, A., Ryan, E.P.: Universal $\lambda$-tracking for
  nonlinearly-perturbed systems in the presence of noise.
\newblock Automatica \textbf{30}(2), 337--346 (1994)

\bibitem{IlchRyan08}
Ilchmann, A., Ryan, E.P.: High-gain control without identification: a survey.
\newblock GAMM Mitt. \textbf{31}(1), 115--125 (2008)

\bibitem{IlchRyan09}
Ilchmann, A., Ryan, E.P.: Performance funnels and tracking control.
\newblock Int. J. Control \textbf{82}(10), 1828--1840 (2009)

\bibitem{IlchRyan02b}
Ilchmann, A., Ryan, E.P., Sangwin, C.J.: Tracking with prescribed transient
  behaviour.
\newblock ESAIM: Control, Optimisation and Calculus of Variations \textbf{7},
  471--493 (2002)

\bibitem{IlchTown93}
Ilchmann, A., Townley, S.B.: Simple adaptive stabilization of high-gain
  stabilizable systems.
\newblock Syst. Control Lett. \textbf{20}(3), 189--198 (1993)

\bibitem{Isid95}
Isidori, A.: Nonlinear Control Systems, 3rd edn.
\newblock Communications and Control Engineering Series. Springer-Verlag,
  Berlin (1995)

\bibitem{LibeTren13b}
Liberzon, D., Trenn, S.: The bang-bang funnel controller for uncertain
  nonlinear systems with arbitrary relative degree.
\newblock {IEEE} Trans. Autom. Control \textbf{58}(12), 3126--3141 (2013)

\bibitem{Muel09a}
Mueller, M.: Normal form for linear systems with respect to its vector relative
  degree.
\newblock Linear Algebra Appl. \textbf{430}(4), 1292--1312 (2009)

\bibitem{NicoTorn89}
Nicosia, S., Tornamb{\`{e}}, A.: High-gain observers in the state and parameter
  estimation of robots having elastic joints.
\newblock Syst. Control Lett. \textbf{13}(4), 331--337 (1989)

\bibitem{SenfPaug14}
Senfelds, A., Paugurs, A.: Electrical drive {DC} link power flow control with
  adaptive approach.
\newblock In: Proc. 55th Int. Sci. Conf. Power Electr. Engg. Riga Techn. Univ.,
  Riga, Latvia, pp. 30--33 (2014)

\bibitem{TrenStoo01}
Trentelman, H.L., Stoorvogel, A.A., Hautus, M.L.J.: Control Theory for Linear
  Systems.
\newblock Communications and Control Engineering. Springer-Verlag, London
  (2001).
\newblock \doi{10.1007/978-1-4471-0339-4}

\end{thebibliography}


\printindex

\end{document}